\pdfoutput=1
\RequirePackage{ifpdf}
\ifpdf 
\documentclass[pdftex]{sigma}
\else
\documentclass{sigma}
\fi

\numberwithin{equation}{section}

\newtheorem{Theorem}{Theorem}[section]
\newtheorem*{Theorem*}{Theorem}
\newtheorem*{TheoremA}{Theorem~A}
\newtheorem*{TheoremB}{Theorem~B}
\newtheorem{Corollary}[Theorem]{Corollary}
\newtheorem{Lemma}[Theorem]{Lemma}
\newtheorem{Proposition}[Theorem]{Proposition}
 { \theoremstyle{definition}
\newtheorem{Definition}[Theorem]{Definition}

\newtheorem{Remark}[Theorem]{Remark} }

\begin{document}

\allowdisplaybreaks

\newcommand{\arXivNumber}{2212.01695}

\renewcommand{\PaperNumber}{067}

\FirstPageHeading

\ShortArticleName{Real Slices of ${\rm SL}(r,\mathbb{C})$-Opers}

\ArticleName{Real Slices of $\boldsymbol{{\rm SL}(r,\mathbb{C})}$-Opers}

\Author{Indranil BISWAS~$^{\rm a}$, Sebastian HELLER~$^{\rm b}$ and Laura P.~SCHAPOSNIK~$^{\rm c}$}

\AuthorNameForHeading{I.~Biswas, S.~Heller and L.P.~Schaposnik}

\Address{$^{\rm a)}$~Department of Mathematics, Shiv Nadar University,\\
\hphantom{$^{\rm a)}$}~NH91, Tehsil Dadri, Greater Noida, Uttar Pradesh 201314, India}
\EmailD{\href{mailto:indranil.biswas@snu.edu.in}{indranil.biswas@snu.edu.in}}

\Address{$^{\rm b)}$~Beijing Institute of Mathematical Sciences and Applications, Beijing 101408, P.R.~China}
\EmailD{\href{mailto:sheller@bimsa.cn}{sheller@bimsa.cn}}

\Address{$^{\rm c)}$~Department of Mathematics, Statistics, and Computer Science,\\
\hphantom{$^{\rm c)}$}~University of Illinois at Chicago, 851 S Morgan St, Chicago, IL 60607, USA}
\EmailD{\href{mailto:schapos@uic.edu}{schapos@uic.edu}}

\ArticleDates{Received April 12, 2023, in final form September 05, 2023; Published online September 16, 2023}

\Abstract{Through the action of an anti-holomorphic involution $\sigma$ (a real structure) on a~Riemann surface $X$, we consider the induced actions on ${\rm SL}(r,\mathbb{C})$-opers and study the real slices fixed by such actions. By constructing this involution for different descriptions of the space of ${\rm SL}(r,\mathbb{C})$-opers, we are able to give a natural parametrization of the fixed point locus via differentials on the Riemann surface, which in turn allows us to study their geometric properties.}


\Classification{14H60; 33C80; 53A55}

\section{Introduction}\label{intro.}

Anti-holomorphic involutions on Riemann surfaces have long been studied, and their induced
actions on associated moduli spaces (such as the moduli space of vector bundles or Higgs
bundles on them) have led to very fruitful results. In particular, by considering the
induced actions on those spaces, one can consider the special subsets of fixed points,
often leading to interesting Lagrangians of the moduli spaces (in the case of Higgs
bundles, see, for instance, \cite{real1,real2}).

The space of ${\rm SL}(r,\mathbb{C})$-opers on $X$, which we denote by ${\rm OSL}_X(r)$, admits several
diverse, but equivalent, descriptions. In the present paper, starting with an anti-holomorphic involution $\sigma$ (a real
structure) on a compact Riemann surface $X$, we construct the induced action of ${\mathbb Z}/2\mathbb Z$
on various spaces associated to $X$:
\begin{itemize}\itemsep=0pt
\item an involution $F$ on the space of projective structures on $X$ (see Lemma \ref{lemmae8}),
\item an involution $\mathcal A_1$ induced by $F$ on the space of projective structures on $X$ times the space of holomorphic differentials \eqref{d4},
\item an involution $\beta$ on
the space of holomorphic differential operators (see Lemma \ref{lem3}),
\item an involution $\mathcal B$ on the space of vector bundles with connections (see Lemma \ref{bigb}),
\item and an involution $\Gamma$ on the
space of equivariant projective embeddings of the universal cover (see Lemma \ref{newinvo}).
\end{itemize}
The fixed point locus in ${\rm OSL}_X(r)$ for the action of ${\mathbb Z}/2\mathbb Z$ on it will be called the
real slice of ${\rm SL}(r,{\mathbb C})$-opers.

Along the paper, we first describe the induced involution from the point of view of just one description of
${\rm OSL}_X(r)$.
Then the space ${\rm OSL}_X(2)$ of ${\rm SL}(2,{\mathbb C})$-opers on $X$ is identified with the space of all
projective structures on $X$. The anti-holomorphic involution $\sigma$ of $X$ produces
an involution of the space of projective structures on $X$. Also $\sigma$ produces a conjugate linear involution of
$H^0\bigl(X, K^{\otimes i}_X\bigr)$ for every $i \geq 0$ by sending any $\omega \in H^0\bigl(X, K^{\otimes i}_X\bigr)$
to $\sigma^*\overline{\omega}$. On the other hand, we have an isomorphism
\[
\Psi\colon\ {\rm OSL}_X(r) \longrightarrow {\rm OSL}_X(2)\times \left(\bigoplus_{i=3}^r H^0\bigl(X, K^{\otimes i}_X\bigr)\right).
\]

\begin{TheoremA}
The above isomorphism $\Psi$ takes the involution of ${\rm OSL}_X(r)$ given by $\sigma$ to the
involution of ${\rm OSL}_X(2)\times \bigl(\bigoplus_{i=3}^r H^0\bigl(X, K^{\otimes i}_X\bigr)\bigr)$ given by the diagonal action of
${\mathbb Z}/2\mathbb Z$ on the product space. In particular, an element
\[
\left(\eta, \bigoplus_{i=3}^r \omega_i\right) \in
{\rm OSL}_X(2)\times \left(\bigoplus_{i=3}^r H^0\bigl(X, K^{\otimes i}_X\bigr)\right)
\]
is in the real slice of ${\rm SL}(r,{\mathbb C})$-opers
if and only if $\eta$ is in the real slice of ${\rm SL}(2,{\mathbb C})$-opers and
$\omega_i = \sigma^*\overline{\omega}_i$ for all $3 \leq i \leq r$.
\end{TheoremA}

The above mentioned ${\mathbb Z}/2\mathbb Z$ actions produce involutions of ${\rm OSL}_X(r)$
from four different points of view, and we finish the paper by showing that all these different involutions of
${\rm OSL}_X(r)$ actually coincide, leading to our second main result.

\begin{TheoremB}
The four involutions $\Gamma$, $\beta$, $\mathcal B$ and $\mathcal A=\Psi^{-1}\circ\mathcal A_1\circ\Psi$ of
${\rm OSL}_X(r)$ coincide.
\end{TheoremB}

\section{Projective structures}

\subsection{Involution of the space of projective structures}

The standard action of the group ${\rm GL}(2,{\mathbb C})$ on ${\mathbb C}^2$ produces an action of
${\rm PGL}(2,{\mathbb C})$ on the complex projective line ${\mathbb C}{\mathbb P}^1$, giving
an identification of ${\rm PGL}(2,{\mathbb C})$ with the holomorphic automorphism group ${\rm Aut}
\bigl({\mathbb C}{\mathbb P}^1\bigr)$ of ${\mathbb C}{\mathbb P}^1$. Through this identification, one can define
a projective structure on an oriented $C^\infty$ surface in the following fashion.

Let $S$ be a compact oriented $C^\infty$ surface of genus $g$, with $g \geq 2$.
Recall that a
coordinate chart on $S$ is a pair $(U, \varphi)$, where $U \subset S$ is an open
subset and $\varphi\colon U \longrightarrow {\mathbb C}{\mathbb P}^1$ is an orientation
preserving $C^\infty$ embedding, and a coordinate atlas is a collection of coordinate charts $\{(U_i,
\varphi_i)\}_{i\in I}$ such that $\bigcup_{i\in I} U_i = S$.

\begin{Definition}
A {\it projective structure} on $S$ is an equivalence class of 
$\{(U_i,\allowbreak \varphi_i)\}_{i\in I}$ such that for all $i, j \in I$, and every nonempty
connected component $U^c_{ij} \subset U_i\cap U_j$, there is an
element
\begin{equation}\label{tc}
T^c_{ij} \in {\rm PGL}(2,{\mathbb C})
\end{equation}
such that the restriction of the maps $\varphi_i\circ \varphi^{-1}_j$ to $\varphi_j(U^c_{ij})$
coincides with the restriction of the automorphism $T^c_{ij} \in {\rm Aut}\bigl({\mathbb C}{\mathbb P}^1\bigr)$. Two
coordinate atlases are called {\it equivalent} if
their union also satisfies the condition \eqref{tc} on the transition functions.
\end{Definition}

{\samepage Note that the condition \eqref{tc} above uniquely
determines the element $T^c_{ij}$. Also a projective structure on $S$ defines
a complex structure on $S$ because the automorphisms of ${\mathbb C}{\mathbb P}^1$ given by the
elements of ${\rm PGL}(2,{\mathbb C})$ are holomorphic.
In what follows we shall denote by
\begin{gather}\label{e0}
\widetilde{\mathcal P}(S)
\end{gather}
the space of all projective structures on $S$.}

A projective structure on a Riemann surface $X$ is a projective structure $P_0$ on the
underlying~$C^\infty$ surface $X_0$ such that the complex structure on $X_0$ given by $P_0$
coincides with the complex structure defining $X$. Consider the orientation reversing involution
\begin{align}
\sigma_0\colon\ {\mathbb C}{\mathbb P}^1 &\longrightarrow {\mathbb C}{\mathbb P}^1,\nonumber\\
\quad(z_1, z_2) &\longmapsto (\overline{z}_1, \overline{z}_2),\label{e1}
\end{align}
for $z_i \in \mathbb C$, and let
\begin{equation}\label{e2}
\sigma\colon\ S \longrightarrow S
\end{equation}
be an orientation reversing diffeomorphism of $S$ such that $\sigma\circ\sigma = {\rm Id}_S$.

\begin{Lemma}\label{invoproj}
Given a projective structure $P$ on $S$ defined by a coordinate atlas
\[
\{(U_i, \varphi_i)\}_{i\in I},\]
the
corresponding atlas
\[
\big\{\bigl(\sigma^{-1}(U_i), \sigma_0\circ\varphi_i\circ\sigma\bigr)\big\}_{i\in
I}
\] also defines a projective structure on $S$, where $\sigma_0$ and $\sigma$ are the involutions
in \eqref{e1} and \eqref{e2} respectively. This produces
an involution of the space $\widetilde{\mathcal P}(S)$ in \eqref{e0}.
\end{Lemma}

\begin{proof}
Given any $i, j \in I$ and any nonempty connected component $U^c_{ij} \subset U_i\cap U_j$,
the composition of maps on $\sigma_0\bigl(\varphi_j(U^c_{ij})\bigr)$
\[
\sigma_0\circ\varphi_i\circ\sigma\circ (\sigma_0\circ\varphi_j\circ\sigma)^{-1} =
\sigma_0\circ \bigl(\varphi_i\circ \varphi^{-1}_j\bigr)\circ\sigma^{-1}_0 =
\sigma_0\circ \bigl(\varphi_i\circ \varphi^{-1}_j\bigr)\circ\sigma_0
\]
is holomorphic, and it is the restriction of the global holomorphic automorphism
$\sigma_0\circ T^c_{ij}\circ\sigma_0$ of ${\mathbb C}{\mathbb P}^1$.
If $\{(U_j, \varphi_j)\}_{j\in J}$ is
equivalent to $\{(U_i, \varphi_i)\}_{i\in I}$, then $\{(\sigma(U_j),
\sigma_0\circ\varphi_j\circ\sigma)\}_{j\in J}$ is also equivalent to $\{(\sigma(U_i),
\sigma_0\circ\varphi_i\circ\sigma)\}_{i\in I}$. Consequently, we get a map on the space of all projective structures on $S$
\begin{equation}\label{e3}
\widetilde{\sigma} \colon\ \widetilde{\mathcal P}(S)
\longrightarrow \widetilde{\mathcal P}(S)
\end{equation}
that sends any projective structure defined by a coordinate
atlas $\{(U_i, \varphi_i)\}_{i\in I}$ to the one defined by the
coordinate atlas $\{(\sigma(U_i), \sigma_0\circ\varphi_i\circ\sigma)\}_{i\in I}$. From the construction of
$\widetilde{\sigma}$ it follows immediately that
 $\widetilde{\sigma}\circ\widetilde{\sigma} = {\rm Id}_{\widetilde{\mathcal
P}(S)}$.
\end{proof}

\begin{Remark}\label{rem1}
The map $\widetilde{\sigma}$ defined in \eqref{e3} is independent
of the choice of the anti-holomorphic involution $\sigma_0$ of ${\mathbb C}{\mathbb P}^1$, but
in general it depends on $\sigma$. Indeed, in its construction, one may replace~$\sigma_0$ (defined in \eqref{e1}) by any other anti-holomorphic involution $\sigma'_0$ of
${\mathbb C}{\mathbb P}^1$ without changing the map $\widetilde{\sigma}$ because
$\sigma'_0$ and $\sigma_0$ differ by an element of ${\rm PGL}(2,{\mathbb C}) =
{\rm Aut}\bigl({\mathbb C}{\mathbb P}^1\bigr)$.
\end{Remark}

We shall denote the group of orientation preserving diffeomorphisms of $S$ by ${\rm Diff}^+(S)$. Let
\[
{\rm Diff}^0(S) \subset {\rm Diff}^+(S)
\]
be the connected component of it consisting of diffeomorphisms homotopic to the
identity map of $S$; it is a normal subgroup of ${\rm Diff}^+(S)$. The
mapping class group can then be seen as the quotient
\[
{\rm MCG}(S) := {\rm Diff}^+(S)/{\rm Diff}^0(S).
\]
Moreover, considering the space ${\mathcal C}(S)$ of complex structures on $S$
compatible with its $C^\infty$ structure and orientation, the Teichm\"uller space can be seen as the quotient
\begin{equation}\label{e5}
{\mathcal T}(S) := {\mathcal C}(S)/{\rm Diff}^0(S),
\end{equation}
which is a complex manifold of dimension $3(g-1)$.

The group ${\rm Diff}^+(S)$ acts on $\widetilde{\mathcal P}(S)$ as follows. Let $P$ be a projective structure on
$S$ given by~$\{(U_i, \varphi_i)\}_{i\in I}$. Then the action of any $\psi \in {\rm Diff}^+(S)$ sends it to
the projective structure $\psi\cdot P$ defined by $\big\{\bigl(\psi(U_i), \varphi_i\circ\psi^{-1}\bigr)\big\}_{i\in I}$.
Consider the corresponding quotient space
\begin{equation}\label{e6}
{\mathcal P}(S) :=
\widetilde{\mathcal P}(S)/{\rm Diff}^0(S),
\end{equation}
which is a complex manifold of dimension $\dim {\mathcal P}(S) = 6(g-1) = 2\dim {\mathcal T}(S)$.
Since a projective structure on $S$ produces a complex structure on $S$, we have a map
\begin{equation}\label{e7}
\Phi\colon\ {\mathcal P}(S) \longrightarrow {\mathcal T}(S),
\end{equation}
where ${\mathcal P}(S)$ and ${\mathcal T}(S)$ are constructed in \eqref{e6} and \eqref{e5}, respectively.

\begin{Remark} Any Riemann surface admits a projective structure.
For example, there is a~natural projective structure given by the uniformization theorem. Consequently, the map
$\Phi$ in \eqref{e7} is surjective. In fact, $\Phi$ is a holomorphic map, and ${\mathcal
P}(S)$ is a holomorphic torsor over ${\mathcal T}(S)$ for the holomorphic cotangent bundle
$T^*{\mathcal T}(S)$ (see \cite{Gu, He, Hu}).
\end{Remark}

\begin{Lemma}\label{lemmae8}
The map $\widetilde{\sigma}$ in \eqref{e3} produces an anti-holomorphic
involution
\begin{equation}\label{e8}
F\colon\ {\mathcal P}(S) \longrightarrow {\mathcal P}(S).
\end{equation}
\end{Lemma}

\begin{proof}
Since ${\rm Diff}^0(S)$ is the connected component of ${\rm Diff}(S)$ containing the identity element, it is a normal subgroup
of ${\rm Diff}(S)$. Moreover, ${\rm Diff}^+(S)$ is also a normal subgroup of ${\rm Diff}(S)$, and~${\rm Diff}(S)/{\rm
Diff}^+(S) = {\mathbb Z}/2{\mathbb Z}$. Take any $\psi \in {\rm Diff}^0(S)$ and any projective structure $P$ on
$S$ given by $\{(U_i, \varphi_i)\}_{i\in I}$. Then the projective structure $\widetilde{\sigma}(\psi\cdot P)$ is given
by $\big\{\bigl(\sigma\circ \psi(U_i), \sigma_0\circ\varphi_i\circ\psi^{-1}\circ\sigma\bigr)\big\}_{i\in I}$. Therefore, we have
\[
\widetilde{\sigma}(\psi\cdot P) = (\sigma\circ\psi\circ\sigma)\cdot\widetilde{\sigma}(P).
\]
Since $\sigma\circ\psi\circ\sigma \in {\rm Diff}^0(S)$ (recall that ${\rm Diff}^0(S)$ is a normal subgroup), we
conclude $\widetilde{\sigma}$ produces a~self-map $F$ of ${\mathcal P}(S)$. This map $F$ is evidently an involution.
Also $F$ is clearly anti-holomorphic.\looseness=1
\end{proof}

For any $C^\infty$ complex structure $J\colon TS \longrightarrow TS$ on $S$ compatible with its orientation,
the pullback $-\sigma^*J$
is a $C^\infty$ complex structure on $S$ compatible with its orientation, and thus there is an induced anti-holomorphic involution
\begin{align}
f\colon\ {\mathcal T}(S) &\longrightarrow {\mathcal T}(S),\nonumber \\
J &\longmapsto -\sigma^*J. \label{e9}
\end{align}
For $f$ and $F$ the involutions constructed in \eqref{e9} and \eqref{e8}, respectively, one can see that
\begin{equation}\label{ec}
\Phi\circ F = f\circ\Phi,
\end{equation}
where $\Phi$ is the projection in \eqref{e7}. Given any Riemann surface $X \!\!\in\! {\mathcal T}(S)$, the fiber
${\Phi^{-1}(X)\! \subset \!{\mathcal P}(S)}$ is an affine space for
$H^0\bigl(X, K^{\otimes 2}_X\bigr)$, where $K_X$ is the holomorphic cotangent bundle
of $X$ \cite{ Gu, Hu}. Moreover, using the isomorphism
\[
T^*_X {\mathcal T}(S) = H^0\bigl(X, K^{\otimes 2}_X\bigr),
\]
${\mathcal P}(S)$ is realized as a holomorphic affine bundle over
${\mathcal T}(S)$ for the holomorphic vector bundle~$T^*{\mathcal T}(S)$ \cite{Gu,Hu}.
The involution $F$ satisfies the equation
\begin{equation}\label{fe}
F(P+\omega) = F(P)+ \sigma^*\overline{\omega}
\end{equation}
for all $P \in \Phi^{-1}(X)$ and $\omega \in H^0\bigl(X, K^{\otimes 2}_X\bigr)$, where
$\sigma^*\overline{\omega}$ is the corresponding holomorphic quadratic differential on the
Riemann surface $f(X) \in {\mathcal T}(S)$.

\section{Real slices of the character variety and Teichm\"uller space}

In what follows, we shall study the induced real slices of the character variety and Teichm\"uller space obtained through the
anti-holomorphic involutions introduced in the previous section.

\subsection{The character variety and symplectic form}

As before, $S$ is a compact oriented surface of genus at least two.
Given a base point $x_0 \in S$, recall that a homomorphism $\rho\colon \pi_1(S, x_0) \longrightarrow
\operatorname{PSL}(2,{\mathbb C})$ is called irreducible if the image of $\rho$ is not contained in any
Borel subgroup of $\operatorname{PSL}(2,{\mathbb C})$.
Let
\[
\operatorname{Hom}(\pi_1(S, x_0), \operatorname{PSL}(2,{\mathbb C}))^{\rm ir} \subset
\operatorname{Hom}(\pi_1(S, x_0), \operatorname{PSL}(2,{\mathbb C}))
\]
be the space of all irreducible homomorphisms.

\begin{Definition}
The {\it irreducible ${\rm PSL}(2,{\mathbb C})$-character variety of $S$} is the quotient space
\begin{equation}\label{e11}
{\mathcal R}_2(S) := \operatorname{Hom}\left(\pi_1(S, x_0), \operatorname{PSL}(2,{\mathbb C})\right)^{\rm ir}/\operatorname{PSL}(2,{\mathbb C})
\end{equation}
for the action of $\operatorname{PSL}(2,{\mathbb C})$ on $\operatorname{Hom}\left(\pi_1(S, x_0), \operatorname{PSL}(2,{\mathbb
C})\right)^{\rm ir}$ given by the conjugation action of~$\operatorname{PSL}(2,{\mathbb C})$ on itself.
\end{Definition}

It should be mentioned that ${\mathcal R}_2(S)$ is independent of the choice of the base point $x_0$.
In fact, ${\mathcal R}_2(S)$ parametrizes the isomorphism classes of flat $\operatorname{PSL}(2,{\mathbb C})$-connections
on $S$.

The space ${\mathcal R}_2(S)$ has two connected components, which are parametrized
by the second Stiefel--Whitney class of the projective bundles associated to the
$\operatorname{PSL}(2,{\mathbb C})$-bundles on $S$. Also each connected component
of ${\mathcal R}_2(S)$ is an irreducible complex affine
variety of dimension $6(g-1)$. We note that its algebraic structure is given by the algebraic structure of
$\operatorname{PSL}(2,{\mathbb C})$ and the fact that the group $\pi_1(S, x_0)$ is
finitely presented. Moreover, the character variety ${\mathcal R}_2(S)$ is a~smooth orbifold
and is equipped with an algebraic symplectic structure \cite{ AB, Go}, which we shall~denote~by\looseness=1
\begin{equation}\label{e13}
\Theta_2 \in H^0\bigl({\mathcal R}_2(S), \Omega^2_{{\mathcal R}_2(S)}\bigr).
\end{equation}

The involution $\sigma$ of $S$ in \eqref{e2} induces an algebraic involution
\begin{align*}
I_2\colon\ {\mathcal R}_2(S) &\longrightarrow {\mathcal R}_2(S), \\
(E, \nabla)&\longmapsto (\sigma^*E, \sigma^*{\nabla}),
\end{align*}
where $\nabla$ is a flat connection on a principal $\operatorname{PSL}(2,{\mathbb C})$-bundle $E$,
equivalently, $I_2$ sends any homomorphism
\[
\rho\colon\
\pi_1(S, x_0) \longrightarrow \operatorname{PSL}(2, {\mathbb C})
\]
 to the homomorphism
$\pi_1(S, \sigma(x_0)) \longrightarrow \operatorname{PSL}(2,{\mathbb C})$ that maps any
$\gamma \in \pi_1(S, \sigma(x_0))$ to
\[
\rho(\sigma(\gamma))
\in \operatorname{PSL}(2,{\mathbb C}).\]
The anti-holomorphic involution $A \longmapsto \overline{A}$
of $\operatorname{PSL}(2,{\mathbb C})$ defines an anti-holomorphic involution of the representation space
$\operatorname{Hom}\left(\pi_1(S, x_0), \operatorname{PSL}(2,{\mathbb C})\right)^{\rm ir}$; this anti-holomorphic involution
\[
\operatorname{Hom}(\pi_1(S, x_0), \operatorname{PSL}(2,{\mathbb C}))^{\rm ir} \longrightarrow
 \operatorname{Hom}(\pi_1(S, x_0), \operatorname{PSL}(2,{\mathbb C}))^{\rm ir}
 \] in turn produces an
anti-holomorphic involution
\begin{equation}\label{eb}
b\colon\ {\mathcal R}_2(S) \longrightarrow {\mathcal R}_2(S)
\end{equation}
of ${\mathcal R}_2(S)$. It is straightforward to check that $b\circ I_2 = I_2\circ b$,
and hence the composition
\begin{equation}\label{e14}
h := b\circ I_2\colon\ {\mathcal R}_2(S) \longrightarrow {\mathcal R}_2(S)
\end{equation}
is an anti-holomorphic involution of ${\mathcal R}_2(S)$ (recall that the map $I_2$
is holomorphic).

\begin{Lemma}\label{lem1}
The symplectic form $\Theta_2$ in \eqref{e13} satisfies the equation
\[
h^*\Theta_2 = -\overline{\Theta}_2,
\]
where $h$ is the involution in \eqref{e14}.
\end{Lemma}

\begin{proof}
We have $I^*_2\Theta_2 = -\Theta_2$, because $\sigma$ is an orientation reversing involution.
We also have~$b^*\Theta_2 = \overline{\Theta}_2$. In view of
\eqref{e14}, these two together
imply that $h^*\Theta_2 = -\overline{\Theta}_2$.
\end{proof}

Giving a projective structure on a Riemann surface $X$ is equivalent to giving a
holomorphic fiber bundle $p\colon Z \longrightarrow X$, a holomorphic connection
$\nabla$ on the fiber bundle and a holomorphic section $s\colon X \longrightarrow Z$
of $p$, such that
\begin{itemize}\itemsep=0pt
\item the fibers of $p$ are isomorphic to ${\mathbb C}{\mathbb P}^1$, and

\item $s$ is transversal to the horizontal distribution for the connection $\nabla$
\end{itemize}
(see \cite{Gu,Hu} for details). Therefore the fiber bundle $Z$ produces a holomorphic principal $\operatorname{PSL}(2,
{\mathbb C})$-bundle $\mathcal P$ over $X$; the fiber of $\mathcal P$ over any $x \in X$ is the space of
all holomorphic isomorphisms from ${\mathbb C}{\mathbb P}^1$ to the fiber $p^{-1}(x)$.
The connection $\nabla$ produces a holomorphic connection on $\mathcal P$. The monodromy representation
$\pi_1(X, x_0) \longrightarrow \operatorname{PSL}(2, {\mathbb C})$ for $\nabla$ is irreducible.

\begin{Lemma}
The symplectic form $\Theta_2$ on the character variety ${\mathcal R}_2(S)$ defined in
\eqref{e13} induces a natural symplectic form $\Theta_{\mathcal P}$ on the
quotient space of projective structures ${\mathcal P}(S)$ introduced in~\eqref{e6}.
\end{Lemma}

\begin{proof}
We have seen above that a projective structure on $S$ defines a flat
$\operatorname{PSL}(2,{\mathbb C})$-connection on $S$. Therefore, assigning
the monodromy representation to the flat connections we get a map
\begin{equation}\label{e12}
\mu\colon\ {\mathcal P}(S) \longrightarrow {\mathcal R}_2(S),
\end{equation}
from the quotient space of projective structures ${\mathcal P}(S)$ introduced in \eqref{e6} to
the $\operatorname{PSL}(2,{\mathbb C})$-character variety ${\mathcal R}_2(S)$ constructed in \eqref{e11}.
It is known that $\mu$ is a local biholomorphism (see~\cite[p.~272]{Hu} and \cite{He}) and thus the
pullback
\begin{equation}\label{e16}
\Theta_{\mathcal P} := \mu^*\Theta_2 \in H^0\bigl({\mathcal P}(S), \Omega^2_{{\mathcal P}(S)}\bigr)
\end{equation}
of the symplectic form in \eqref{e13} is a holomorphic symplectic form
on ${\mathcal P}(S)$.\end{proof}

\begin{Remark}From the constructions of $F$ and $h$ in \eqref{e8} and \eqref{e14} respectively, it follows that
\begin{equation}\label{e17}
\mu \circ F = h\circ\mu,
\end{equation}
where $\mu$ is the monodromy map in \eqref{e12}.\end{Remark}

\begin{Corollary}\label{cor1}
The symplectic form $\Theta_{\mathcal P}$ in \eqref{e16} satisfies the equation
\[
F^*\Theta_{\mathcal P} = -\overline{\Theta}_{\mathcal P}.
\]
\end{Corollary}

\begin{proof}
{}From \eqref{e17} we have $F^*\Theta_{\mathcal P} = F^*\mu^*\Theta_2 = \mu^*h^*
\Theta_2$. Therefore, from Lemma \ref{lem1} and \eqref{e16} it follows that $F^*
\Theta_{\mathcal P} = -\mu^*\overline{\Theta}_2 = -\overline{\Theta}_{\mathcal P}$.
\end{proof}

\subsection{Real projective structures}

A projective structure $P \in {\mathcal P}(S)$ is called \textit{real} if it is a fixed point of the anti-holomorphic involution
\[F\colon\ {\mathcal P}(S) \longrightarrow {\mathcal P}(S)\]
introduced in \eqref{e8}. Moreover, given a real projective structure $P \in {\mathcal P}(S)$, from \eqref{ec} one obtains an
induced real point $\Phi(P) \in {\mathcal T}(S)$ with respect to the anti-holomorphic involution $f$, meaning
\[
f(\Phi(P)) = \Phi(P).
\]
The following lemma is a sort of converse of it.

\begin{Lemma}\label{lem2}
Given any $J \in {\mathcal T}(S)$ such that $f(J) = J$, there exists a real projective structure~$P$
such that $\Phi(P) = J$.
\end{Lemma}

\begin{proof}
Let $P$ be the projective structure given by the uniformization of the Riemann surface~$(S, J)$. Then it
is evident that $P$ is real and satisfies the equation $\Phi(P) = J$.
\end{proof}

In what follows we shall denote by
${\mathcal T}(S)^f \subset {\mathcal T}(S)$ and $ {\mathcal P}(S)^F \subset {\mathcal P}(S)$
 the fixed point loci for the involutions $f$ and $F$ respectively. The fixed point locus ${\mathcal P}(S)^F$ \big(respectively, ${\mathcal T}(S)^f$\big) is a~$C^\infty$ real
submanifold of ${\mathcal P}(S)$ (respectively, ${\mathcal T}(S)$) of dimension $6(g-1)$
(respectively, $3(g-1)$). Let\looseness=1
\begin{equation}\label{ep0}
\Phi_0\colon\ {\mathcal P}(S)^F \longrightarrow {\mathcal T}(S)^f
\end{equation}
be the restriction of the map $\Phi$ in \eqref{e7}; as noted above, from \eqref{ec} it follows immediately that
$\Phi$ maps ${\mathcal P}(S)^F$ to ${\mathcal T}(S)^f$. From Lemma \ref{lem2}, it follows that the map
$\Phi_0$ is surjective.

The differential ${\rm d}f$ of the anti-holomorphic involution $f$ of ${\mathcal T}(S)$ produces an anti-holomorphic
involution of the holomorphic cotangent bundle $T^*{\mathcal T}(S)$. Let
\[
T^*{\mathcal T}(S)^f \subset T^*{\mathcal T}(S)
\]
be the fixed point locus of this involution. This $T^*{\mathcal T}(S)^f$ is a $C^\infty$ real
vector bundle over ${\mathcal T}(S)^f$. We note that $T^*{\mathcal T}(S)^f$ is identified with the (real) cotangent
bundle $T^*\big({\mathcal T}(S)^f\big)$ of ${\mathcal T}(S)^f$.

\begin{Lemma} \label{affinebun}The space ${\mathcal P}(S)^F$ is an affine bundle over
${\mathcal T}(S)^f$ for the real vector bundle $T^*{\mathcal T}(S)^f = T^*\bigl({\mathcal T}(S)^f\bigr)$.
\end{Lemma}

\begin{proof}
Given any $J \in {\mathcal T}(S)^f$, let $X$ be the corresponding Riemann surface $(S, J)$, and denote the
holomorphic cotangent bundle of $X$ by $K_X$. Let
\[
H^0\bigl(X, K^{\otimes 2}_X\bigr)^\sigma \subset H^0\bigl(X, K^{\otimes 2}_X\bigr)
\]
be the fixed point locus of the conjugate linear involution $w \longmapsto \sigma^*\overline{w}$
of $H^0\bigl(X, K^{\otimes 2}_X\bigr)$. The fiber over $X$ of the map $\Phi_0\colon \mathcal P(S)^F
 \longrightarrow {\mathcal T}(S)^f$ in \eqref{ep0} is actually an affine space for
the real vector space $H^0\bigl(X, K^{\otimes 2}_X\bigr)^\sigma$. Indeed, this follows
immediately from the combination of the fact that
the space of all projective structures on the Riemann surface $X$ is an affine
space, for the complex vector space $H^0\bigl(X, K^{\otimes 2}_X\bigr)$, and \eqref{fe}.
We note that
\[
T^*_X \bigl({\mathcal T}(S)^f\bigr) = H^0\bigl(X, K^{\otimes 2}_X\bigr)^\sigma.\]
Consequently, ${\mathcal P}(S)^F$ is an affine bundle over
${\mathcal T}(S)^f$ for the real vector bundle $T^*\bigl({\mathcal T}(S)^f\bigr)$ as needed.
\end{proof}

\begin{Proposition}\label{prop1}
The manifold ${\mathcal P}(S)^F$ is connected, and it is homeomorphic to the Euclidean space.
\end{Proposition}

\begin{proof}
The manifold ${\mathcal T}(S)^f$ is connected, and it is homeomorphic to the Euclidean space
(see~\cite[p.~122]{Ea} and \cite{Kr}). On the other hand, ${\mathcal P}(S)^F$ is an affine bundle over
${\mathcal T}(S)^f$ for the vector bundle $T^*{\mathcal T}(S)^f$ from Lemma \ref{affinebun}, and hence
the proposition follows.
\end{proof}

Since $F$ is an anti-holomorphic involution, Corollary \ref{cor1} gives the following.

\begin{Corollary}\label{cor2}
Let $\Theta^F_{\mathcal P}$ be the restriction of the holomorphic symplectic form $\Theta_{\mathcal P}$ to
${\mathcal P}(S)^F$. Then the real part $\operatorname{Re}\bigl(\Theta^F_{\mathcal P}\bigr)$ satisfies the equation
\[
\operatorname{Re}\bigl(\Theta^F_{\mathcal P}\bigr) = 0,
\]
and the imaginary part $\operatorname{Im}\bigl(\Theta^F_{\mathcal P}\bigr)$ is a real symplectic form on ${\mathcal P}(S)^F$.
\end{Corollary}

\subsection{Projective structures with Fuchsian monodromy}

It should be mentioned that the notion of {\em real projective structure} is used in the literature also in a
different context than ours above, see, for example, \cite{Fa} and also \cite{Hel, Tak,Te}. In these papers, a real projective structure
is a projective structure with Fuchsian monodromy.
In our notation, projective structures with Fuchsian monodromy are given by certain components of
\[\{P\in\mathcal P(S) \mid \mu(P) = b\circ\mu(P)\},\]
where $b$ is the involution in \eqref{eb}.
Examples are given by the uniformization of (any) Riemann surface, but there are others obtained by grafting. In fact, it was shown by Goldman \cite{Go2} that any projective structure with Fuchsian monodromy is obtained by the process of $2\pi$-grafting of the projective structure obtained from
the uniformization of an appropriate Riemann surface.

\section[Real and quaternionic SL(r,C)-opers]{Real and quaternionic $\boldsymbol{{\rm SL}(r,\mathbb{C})}$-opers}\label{realoper}

Having studied anti-holomorphic involutions on character varieties, projective structures and Teichm\"uller space, we shall now consider
the induced actions on different moduli spaces related to them, and in particular, on ${\rm SL}(r,\mathbb{C})$-opers.

\subsection{Real and quaternionic bundles}\label{invo}
As in the previous sections, consider a pair $(X, \sigma_X)$ of a compact connected Riemann surface $X$ of genus
$ g \geq 2$, and an anti-holomorphic involution
\begin{equation}\label{e18}
\sigma_X\colon\ X \longrightarrow X.
\end{equation}
For any complex vector bundle $E$ on $X$, let $\overline{E}$ denote the
complex vector bundle on $X$ whose underlying real vector bundle is identified with $E$, while the multiplication
by $\sqrt{-1}$ on $\overline{E}$ coincides with the multiplication by $-\sqrt{-1}$ on $E$.

\begin{Lemma}For any holomorphic vector bundle $V$ on $X$, there is a holomorphic structure
on the $C^\infty$ vector bundle $\sigma^*_X \overline{V}$ given by the holomorphic structure on $V$.
\end{Lemma}

\begin{proof}
We shall construct the holomorphic structure on $\sigma^*_X \overline{V}$ in the following way. Consider a
$C^\infty$ section $s$ of $\sigma^*_X\overline{V}$ defined over an open subset
$U \subset X$. So $s$ gives a $C^\infty$ section of $\overline{V}$ over~$\sigma_X(U)$.
As $V = \overline{V}$ as real bundles, this section of $\overline{V}$ over $\sigma_X(U)$
gives a $C^\infty$ section of $V$, which we denote by $\widetilde{s}$, over $\sigma_X(U)$.
The holomorphic structure on $\sigma^*_X \overline{V}$ is defined by the following condition:
Any section~$s$ of $\sigma^*_X\overline{V}$ as above is holomorphic if and only if the corresponding section~$\widetilde{s}$ of $V$ is holomorphic.\end{proof}

Recall that a theta characteristic on a compact connected Riemann surface $Y$ is a holomorphic line
bundle $K^{1/2}_Y$ on $Y$ such that $K^{1/2}_Y\otimes K^{1/2}_Y$ is holomorphically isomorphic to
the holomorphic cotangent bundle $K_Y$. There are exactly $2^{2k}$ theta characteristics on $Y$, where
$k$ is the genus of~$Y$.

It is known that there are theta characteristics $K^{1/2}_X$ on $X$ such that $\sigma^*_X
\overline{K^{1/2}_X}$ is holomorphically isomorphic to $K^{1/2}_X$ (see \cite[p.~61, Remark
2]{At2}).

Take a holomorphic line bundle $\mathbf L$ on $X$ such that $\sigma^*_X \overline{\mathbf
L}$ is holomorphically isomorphic to $\mathbf L$. For any holomorphic isomorphism $h\colon
{\mathbf L} \longrightarrow \sigma^*_X \overline{\mathbf L}$, the composition of
homomorphisms $\bigl(\sigma^*_X \overline{h}\bigr)\circ h$ is a holomorphic automorphism of ${\mathbf
L}$. Multiplying $h$ by a suitable nonzero complex number, we may ensure that $\bigl(\sigma^*_X
\overline{h}\bigr)\circ h$ is either the identity map ${\rm Id}_{\mathbf L}$ of $\mathbf L$ or
it is $-{\rm Id}_{\mathbf L}$. The holomorphic line bundle $\mathbf L$ is called {\it real}
(respectively, \textit{quaternionic}) if $\bigl(\sigma^*_X \overline{h}\bigr)\circ h$ can be made to
be ${\rm Id}_{\mathbf L}$ (respectively, $-{\rm Id}_{\mathbf L}$); see \cite[p.~206,
Definition 3.3]{BHH}.

Given a theta characteristic $K^{1/2}_X$ on $X$ such that $\sigma^*_X \overline{K^{1/2}_X}$
is holomorphically isomorphic to~$K^{1/2}_X$, since $\text{degree}\bigl(K^{1/2}_X\bigr) = g-1$, the following two hold:
\begin{itemize}\itemsep=0pt
\item[(1)] Assume that $\sigma_X$ does not have any fixed point. If $g$ is even, then $K^{1/2}_X$ must be quaternionic
because $\text{degree}\bigl(K^{1/2}_X\bigr)$ is odd \cite[p.~210, Proposition 4.2]{BHH}.

\item[(2)] Assume that $\sigma_X$ has fixed points. Then $K^{1/2}_X$ must be real (see \cite[p.~210, Section 4.2]{BHH}).
\end{itemize}

\begin{Definition}\label{dtc}
For simplicity, we shall denote the dual line bundle $\bigl(K^{1/2}_X\bigr)^*$ by $\mathcal L$, and thus one has that $\mathcal L$ is
isomorphic to $\sigma^*_X \overline{\mathcal L}$. We note that $\mathcal L$ is real
(respectively, quaternionic) if and only if $K^{1/2}_X$ is real (respectively,
quaternionic).
\end{Definition}

\subsection{Definition of opers}

We shall briefly recall here the construction of the jet bundles $J^m(W)$, $m \geq 0$, of
a holomorphic vector bundle $W$ on $X$. For $i = 1, 2$, let
\begin{equation}\label{z1}
p_i\colon\ X\times X \longrightarrow X
\end{equation}
be the projection to the $i$-th factor, and let
\begin{equation}\label{z2}
\Delta := \{(x, x) \in X\times X \mid x \in X\} \subset X\times X
\end{equation}
be the reduced diagonal divisor. Then for any $m \geq 0$, one can define the associated Jet bundle as
\begin{equation}\label{dj}
J^m(W) := p_{1*}((p^*_2W)/(p^*_2W\otimes
{\mathcal O}_{X\times X}(-(m+1)\Delta))) \longrightarrow X.
\end{equation}
The vector bundle $J^m(W)$ fits in the short exact sequence of holomorphic vector bundles
\begin{equation}\label{ej}
0 \longrightarrow K^{\otimes (m+1)}_X\otimes W \longrightarrow J^{m+1}(W)
 \longrightarrow J^m(W) \longrightarrow 0.
\end{equation}
Using \eqref{dj}, it is straightforward to deduce that there is a natural holomorphic
isomorphism
\begin{equation}\label{dj2}
\sigma^*_X \overline{J^m(W)} \stackrel{\sim}{\longrightarrow}
J^m\bigl(\sigma^*_X\overline{W}\bigr).
\end{equation}

\begin{Lemma}\label{4.7}
Take a holomorphic line bundle $\mathcal L$ on $X$ such that ${\mathcal L}\otimes{\mathcal L} = TX$. The determinant
line bundle $\det {\rm Sym}^{r-1}\bigl(J^1(\mathcal L)\bigr)$ of the $(r-1)$-th symmetric product is
\begin{equation}\label{f1}
\det {\rm Sym}^{r-1}\bigl(J^1(\mathcal L)\bigr) = \bigwedge\nolimits^r {\rm Sym}^{r-1}\bigl(J^1(\mathcal L)\bigr)
 = {\mathcal O}_X.
\end{equation}
\end{Lemma}

\begin{proof}
{}From \eqref{ej}, it follows that the line bundle $\det J^1(\mathcal L) :=
\bigwedge^2 J^1(\mathcal L)$ is
\begin{equation}\label{e4}
\det J^1(\mathcal L) = K_X\otimes {\mathcal L}\otimes {\mathcal L} = K_X\otimes TX = {\mathcal O}_X.
\end{equation}
Since $\det {\rm Sym}^{r-1}\bigl(J^1(\mathcal L)\bigr) = \bigl(\det J^1(\mathcal L)\bigr)^{\otimes r(r-1)/2}$, the
isomorphism in \eqref{f1} follows from \eqref{e4} for every $r \geq 2$.
\end{proof}

Via the notion of jet bundles, we can now recall from
\cite{BD} the definition of an ${\rm SL}(r, {\mathbb C})$-oper on Riemann surfaces.

\begin{Definition}\label{de-o}
An ${\rm SL}(r, {\mathbb C})$-oper on a Riemann surface $X$ of genus
$g \geq 2$ is a holomorphic connection $D$ on
${\rm Sym}^{r-1}\bigl(J^1(\mathcal L)\bigr)$ such that the connection on $\det {\rm
Sym}^{r-1}\bigl(J^1(\mathcal L)\bigr)$ induced by $D$ is the trivial connection on ${\mathcal O}_X$.
\end{Definition}

\begin{Remark}\label{rem-th}
Let ${\mathcal L}'$ be another holomorphic line bundle on $X$ such that
${\mathcal L}'\otimes{\mathcal L}' = TX$. So the holomorphic line bundle $\xi := {\mathcal L}'\otimes {\mathcal L}^*$
satisfies the equation $\xi\otimes\xi = {\mathcal O}_X$. Therefore, $\xi\otimes\xi$ has a unique holomorphic connection
$D_0$ with trivial monodromy. In other words, $D_0$ is the trivial connection. There is a unique holomorphic connection $D_\xi$
on $\xi$ such that the connection on $\xi\otimes\xi$ induced by $D_\xi$ coincides with $D_0$. Using $D_\xi$, there is a canonical
isomorphism
\[
\phi_\xi\colon\ J^1({\mathcal L})\otimes\xi \stackrel{\sim}{\longrightarrow} J^1({\mathcal L}').
\]
To construct $\phi_\xi$, let $\underline{\xi}$ denote the locally constant sheaf on $X$ defined by the sheaf of flat sections of
$\xi$ for the connection $D_\xi$. Consider the natural homomorphism
\[
J^1({\mathcal L})\otimes_{\mathbb C}\underline{\xi} \longrightarrow J^1({\mathcal L}').
\]
This homomorphism extends uniquely to $\phi_\xi$. Note that $\phi_\xi$ produces a holomorphic isomorphism
\begin{gather}\label{phir}
\phi^r_\xi\colon\ {\rm Sym}^{r-1}\bigl(J^1({\mathcal L})\bigr)\otimes \xi^{\otimes r}
 \stackrel{\sim}{\longrightarrow} {\rm Sym}^{r-1}\bigl(J^1({\mathcal L}')\bigr).
\end{gather}
The connection $D_\xi$ on $\xi$ induces a connection $D^r_\xi$ on $\xi^{\otimes r}$. Therefore, the isomorphism $\phi^r_\xi$ in
\eqref{phir} produces a bijection between the space of holomorphic connections on ${\rm Sym}^{r-1}\bigl(J^1({\mathcal L})\bigr)$
and the space of holomorphic connections on ${\rm Sym}^{r-1}\bigl(J^1({\mathcal L}')\bigr)$. Consequently,
the space of holomorphic connections $D$ on ${\rm Sym}^{r-1}\bigl(J^1(\mathcal L)\bigr)$ such that the connection on $\det {\rm
Sym}^{r-1}\bigl(J^1(\mathcal L)\bigr)$ induced by~$D$ is the trivial connection on ${\mathcal O}_X$
(see Definition \ref{de-o}) does not depend on the choice of the line bundle $\mathcal L$.
\end{Remark}

Let ${\mathcal C}_X(r)$ denote the space of all holomorphic connections on ${\rm Sym}^{r-1}\bigl(J^1(\mathcal L)\bigr)$ such that their induced connection on $\det {\rm Sym}^{r-1}\bigl(J^1(\mathcal
L)\bigr)$ is the trivial connection on ${\mathcal O}_X$. Let
\[
\text{Aut}\bigl({\rm
Sym}^{r-1}\bigl(J^1(\mathcal L)\bigr)\bigr)\]
 denote the group of all holomorphic automorphisms of ${\rm
Sym}^{r-1}\bigl(J^1(\mathcal L)\bigr)$. If $D$ is a holomorphic connection on ${\rm
Sym}^{r-1}\bigl(J^1(\mathcal L)\bigr)$ such that the connection on the line bundle
\[
\det {\rm
Sym}^{r-1}\bigl(J^1(\mathcal L)\bigr)\]
 induced by $D$ is the trivial connection, and
\[
A \in \text{Aut}\bigl({\rm Sym}^{r-1}\bigl(J^1(\mathcal L)\bigr)\bigr),
\]
then 
$(A\otimes{\rm
Id}_{K_X})\circ D\circ A^{-1}$ also has the property that the connection on $\det {\rm
Sym}^{r-1}\bigl(J^1(\mathcal L)\bigr)$ induced by it is the trivial connection. Indeed, this follows immediately from
the fact that the holomorphic automorphisms of a line bundle $\xi$ act trivially on the space of all
holomorphic connections on $\xi$.

\begin{Definition}\label{deo}
The moduli space of ${\rm SL}(r, {\mathbb C})$-opers on $X$, which will be denoted by ${\rm OSL}_X(r)$, is defined as follows
\begin{equation}\label{f2}
{\rm OSL}_X(r) := {\mathcal C}_X(r)/\text{Aut}\bigl({\rm
Sym}^{r-1}\bigl(J^1(\mathcal L)\bigr)\bigr).
\end{equation}\end{Definition}

{}Using the isomorphism $\phi^r_\xi$ in \eqref{phir}, we conclude that
\[
\text{Aut}\bigl({\rm Sym}^{r-1}\bigl(J^1({\mathcal L}')\bigr)\bigr)
 = \text{Aut}\bigl({\rm Sym}^{r-1}\bigl(J^1(\mathcal L)\bigr)\bigr).\]
 Therefore, from Remark \ref{rem-th} it follows that
${\rm OSL}_X(r)$ in Definition \ref{deo} is independent of the choice of the line bundle $\mathcal L$.

\begin{Remark}
We need to explain the reason for the assumption that the genus of $X$ is at least two. On ${\mathbb C}{\mathbb P}^1$ there
is no nontrivial ${\rm SL}(r, {\mathbb C})$-oper. If $\text{genus}(X) = 1$, then the underlying holomorphic vector
bundle for an ${\rm SL}(r, {\mathbb C})$-oper is \textit{not} unique up to tensoring with line bundles of order two (unlike
in the case of higher genus Riemann surfaces). For example, when ${\text{genus}(X) = 1}$, both
${\mathcal O}_X\oplus {\mathcal O}_X$ and the unique nontrivial extension of ${\mathcal O}_X$ by ${\mathcal O}_X$ admit an~${\rm SL}(2, {\mathbb C})$-oper structure.
\end{Remark}

\subsection{Differential operators and opers}

For completion, we shall quickly recall the definition of holomorphic differential operators. Given holomorphic
vector bundles $V$ and $W$ on $X$, and a nonnegative integer $d$, define the holomorphic vector bundle
\[
\operatorname{Diff}^d_X(W, V) := V\otimes J^d(W)^* = \operatorname{Hom}\bigl(J^d(W), V\bigr).
\]
A holomorphic differential operator of order $d$ from $W$ to $V$ is a holomorphic
section of the vector bundle $\operatorname{Diff}^d_X(W, V)$.
Consider the homomorphism $K^{\otimes d}_X\otimes W \hookrightarrow J^d(W)$
in \eqref{ej}. Using its dual%
\[
\nu\colon\ J^d(W)^* \longrightarrow \bigl(K^{\otimes d}_X\otimes W\bigr)^*,
\]
we get a surjective homomorphism
\begin{equation}\label{es}
\operatorname{Diff}^d_X(W, V) = V\otimes J^d(W)^*
 \xrightarrow{ {\rm Id}\otimes\nu } V\otimes \bigl(K^{\otimes d}_X\otimes W\bigr)^*
 = (TX)^{\otimes d}\otimes \operatorname{Hom}(W, V),
\end{equation}
which is known as the {\it symbol map}.

The following notation will be used: For a line bundle $\xi$ on $X$ and a positive integer $d$, the line bundle
$(\xi^*)^{\otimes d}$ will be denoted by $\xi^{\otimes -d}$. Also $\xi^0$ will denote the trivial line bundle
${\mathcal O}_X$.

When considered on
${\rm Diff}^r_X\bigl({\mathcal L}^{\otimes (r-1)},
{\mathcal L}^{-r-1}\bigr)$, the symbol map in \eqref{es} sends it to
\[
(TX)^{\otimes r}\otimes \operatorname{Hom}\bigl({\mathcal L}^{\otimes (r-1)}, {\mathcal L}^{-r-1}\bigr)
 = (TX)^{\otimes r}\otimes {\mathcal L}^{-2r} = {\mathcal O}_X,
\]
where we use for a holomorphic line bundle $A$ and a positive integer $d$ the convention ${A^{-d}\!\!=\!(\!A^*\!)^{\otimes d}}$\!.

The following proposition is well known.

\begin{Proposition}\label{f3}
The moduli space of ${\rm SL}(r, {\mathbb C})$-opers on $X$ in \eqref{f2} is isomorphic to the subset
\begin{gather}
\mathcal U \subset H^0\bigl(X, {\rm Diff}^r_X\bigl({\mathcal L}^{\otimes (r-1)},
{\mathcal L}^{-r-1}\bigr)\bigr) = H^0\bigl(X,
{\rm Diff}^r_X\bigl({\mathcal L}^{\otimes (r-1)}, \bigl({\mathcal L}^{\otimes (r+1)}\bigr)^*\bigr)\bigr)\nonumber
\end{gather}
consisting of all differential operators $\delta \in H^0\bigl(X, {\rm Diff}^r_X\bigl({\mathcal
L}^{\otimes (r-1)}, \bigl({\mathcal L}^{\otimes (r+1)}\bigr)^*\bigr)\bigr)$ such that the symbol of $\delta$
is the constant function $1$ on $X$ and the sub-leading term of $\delta$ is vanishing
identically.\looseness=1
\end{Proposition}

\begin{Remark}One should note that the above condition in Proposition \ref{f3} that the sub-leading term of $\delta$ vanishes
identically makes sense intrinsically. Indeed, this condition is equivalent to the condition
that the local system on $X$ given by the sheaf of solutions of $D$ is an ${\rm SL}(r,
{\mathbb C})$-system (meaning the $r$-th exterior product of it is the constant sheaf with
stalk $\mathbb C$).\end{Remark}

We shall briefly describe here the isomorphism
\[{\rm OSL}_X(r) \stackrel{\sim}
{\longrightarrow} \mathcal U \subset H^0\bigl(X, {\rm Diff}^r_X
\bigl({\mathcal L}^{\otimes (r-1)}, \bigl({\mathcal L}^{\otimes (r+1)}\bigr)^*\bigr)\bigr)\]
in Proposition \ref{f3} in more detail. Given a differential operator
\[
\delta \in H^0\bigl(X, {\rm Diff}^r_X\bigl({\mathcal
L}^{\otimes (r-1)}, \bigl({\mathcal L}^{\otimes (r+1)}\bigr)^*\bigr)\bigr)\]
 such that the symbol of $\delta$ is the constant function
$1$ on $X$, its sheaf of solutions produces a~flat vector bundle on $X$. The holomorphic
vector bundle underlying this flat bundle is isomorphic to $J^{r-1}(\mathcal L)$. In other
words, the sheaf of solutions of $\delta$ produces a holomorphic connection $D$ on
$J^{r-1}(\mathcal L)$. The holomorphic vector bundle $J^{r-1}(\mathcal L)$ is
holomorphically isomorphic to~${\rm Sym}^{r-1}\bigl(J^1(\mathcal L)\bigr)$. After fixing a
holomorphic isomorphism between $J^{r-1}(\mathcal L)$ and ${\rm Sym}^{r-1}\bigl(J^1(\mathcal
L)\bigr)$, the holomorphic connection on $D$ on $J^{r-1}(\mathcal L)$ gives a holomorphic
connection on ${\rm Sym}^{r-1}\bigl(J^1(\mathcal L)\bigr)$.

To describe the reverse map
\begin{equation}\label{z1b}
{\rm OSL}_X(r) \longrightarrow H^0\bigl(X, {\rm Diff}^r_X
\bigl({\mathcal L}^{\otimes (r-1)}, \bigl({\mathcal L}^{\otimes (r+1)}\bigr)^*\bigr)\bigr),
\end{equation}
take a holomorphic connection $D$ on ${\rm Sym}^{r-1}\bigl(J^1(\mathcal
L)\bigr)$ such that the connection on the line bundle $\det {\rm Sym}^{r-1}\bigl(J^1(\mathcal L)\bigr)$ induced by $D$ is
the trivial connection on ${\mathcal O}_X$. Using $D$, we will construct a~holomorphic
isomorphism between ${\rm Sym}^{r-1}\bigl(J^1(\mathcal L)\bigr)$ and $J^{r-1}(\mathcal L)$. Consider
the projection
\[
\psi\colon\ {\rm Sym}^{r-1}\bigl(J^1(\mathcal L)\bigr) \longrightarrow {\rm Sym}^{r-1}\bigl(J^0(\mathcal L)\bigr) =
{\mathcal L}^{\otimes (r-1)},
\]
which is the $(r-1)$-th symmetric power of the natural projection $J^1(\mathcal L)
\longrightarrow \mathcal L$ (see the exact sequence in \eqref{ej}). Take any $x \in X$ and $v \in {\rm
Sym}^{r-1}\bigl(J^1(\mathcal L)\bigr)_x$. Let $\widetilde{v}$ denote the unique flat section of ${\rm
Sym}^{r-1}\bigl(J^1(\mathcal L)\bigr)\big\vert_U$, for the connection $D$, defined on a simply
connected open neighborhood $x \in U \subset X$, such that $\widetilde{v}(x) =
v$. Restricting $\psi(\widetilde{v})$ to the $(r-1)$-th order infinitesimal neighborhood of
$x$ we get an element $\widetilde{v}' \in J^{r-1}\bigl({\mathcal L}^{\otimes (r-1)}\bigr)_x$.
Consequently, we have a holomorphic map
\begin{align}
\widetilde{\psi}\colon\ {\rm Sym}^{r-1}\bigl(J^1(\mathcal L)\bigr) & \longrightarrow J^{r-1}\bigl({\mathcal L}^{\otimes (r-1)}\bigr),\nonumber\\
v& \longmapsto \widetilde{v}',\label{psn}
\end{align}
which is in fact an isomorphism. Let $\bigl(\widetilde{\psi}^{-1}\bigr)^*D$
be the holomorphic connection on the vector bundle $J^{r-1}\bigl({\mathcal L}^{\otimes (r-1)}\bigr)$ obtained by pulling
back the connection $D$ on ${\rm Sym}^{r-1}\bigl(J^1(\mathcal L)\bigr)$ using $\widetilde{\psi}^{-1}$.
Then there is a unique
\begin{equation}\label{a2}
\delta \in H^0\bigl(X, {\rm Diff}^r_X\bigl({\mathcal L}^{\otimes (r-1)},
\bigl({\mathcal L}^{\otimes (r+1)}\bigr)^*\bigr)\bigr)
\end{equation}
{\samepage such that
\begin{itemize}\itemsep=0pt
\item the symbol of $\delta$ is the constant function $1$ on $X$, and

\item the flat vector bundle on $X$ given by the sheaf of solutions of $\delta$ is isomorphic
to $\bigl(\widetilde{\psi}^{-1}\bigr)^*D$ on $J^{r-1}\bigl({\mathcal L}^{\otimes (r-1)}\bigr)$ equipped with the above
flat connection.
\end{itemize}}
Since $\bigl(\widetilde{\psi}^{-1}\bigr)^*D$ is a holomorphic $\text{SL}(r,{\mathbb C})$ connection, the
sub-leading term of $\delta$ vanishes identically as required.

\subsection[Real and quaternionic slices of SL(r,C)-opers]{Real and quaternionic slices of $\boldsymbol{{\rm SL}(r,\mathbb{C})}$-opers}

As before, $\mathcal L$ is a holomorphic line bundle on $X$ with ${\mathcal L}\otimes {\mathcal L} = TX$.
The Riemann surface is now equipped with an anti-holomorphic involution $\sigma_X$.

In order to define and study real and quaternionic slices of ${\rm SL}(r,\mathbb{C})$-opers, we shall consider certain
compatibility conditions with the anti-holomorphic involutions studied in Section \ref{invo}.
We shall begin by fixing a holomorphic isomorphism of line bundles
\begin{equation}\label{f4}
\widehat{\sigma}\colon\ {\mathcal L} \longrightarrow \sigma^*_X \overline{\mathcal L}
\end{equation}
such that $\bigl(\sigma^*_X \overline{\widehat{\sigma}}\bigr)\circ\widehat{\sigma} \in \pm {\rm Id}_{\mathcal L}$,
where $\sigma_X$ is the anti-holomorphic involution in \eqref{e18}. Note that such pairs $\bigl({\mathcal K}, \widehat{\sigma}\bigr)$
exist. The isomorphism $\widehat{\sigma}$ in \eqref{f4} above produces a holomorphic isomorphism%
\begin{gather}\label{s2}
\widehat{\sigma}^m_k\colon\ J^m\bigl({\mathcal L}^k\bigr) \longrightarrow
J^m\bigl(\sigma^*_X \overline{\mathcal L}^k\bigr)
 = \sigma^*_X \overline{J^m\bigl({\mathcal L}^k\bigr)}
\end{gather}
for all $m \geq 0$ and $k \in \mathbb Z$, where the second isomorphism in \eqref{s2} is given by
\eqref{dj2}. Moreover, since $\bigl(\sigma^*_X \overline{\widehat{\sigma}}\bigr)
\circ\widehat{\sigma} \in \pm {\rm Id}_{\mathcal L}$ (see \eqref{f4}), it follows that
\begin{equation}\label{s2a}
\bigl(\sigma^*_X \overline{\widehat{\sigma}^m_k}\bigr)\circ\widehat{\sigma}^m_k \in \pm {\rm Id}_{J^m({\mathcal L}^k)}.
\end{equation}
The homomorphism $\widehat{\sigma}^1_1$ defined through \eqref{s2} produces, in turn, a
holomorphic isomorphism of symmetric products
\begin{equation}\label{f5}
{\rm Sym}^{r-1}\bigl(\widehat{\sigma}^1_1\bigr)\colon\ {\rm Sym}^{r-1}\bigl(J^1(\mathcal L)\bigr)
\longrightarrow {\rm Sym}^{r-1}\bigl(\sigma^*_X
\overline{J^1(\mathcal L)}\bigr) = \sigma^*_X\overline{{\rm Sym}^{r-1}\bigl(J^1({\mathcal L})\bigr)}.
\end{equation}
{}From \eqref{s2a} for $m = 1$, we deduce that
\[
\bigl(\sigma^*_X \overline{{\rm Sym}^{r-1}\bigl(\widehat{\sigma}^1_1\bigr)}\bigr)\circ{\rm Sym}^{r-1}\bigl(\widehat{\sigma}^1_1\bigr) \in
\pm {\rm Id}_{{\rm Sym}^{r-1}(J^1(\mathcal L))}.
\]
Observe that $\bigl(\sigma^*_X \overline{{\rm Sym}^{r-1}\bigl(\widehat{\sigma}^1_1\bigr)}\bigr)\circ
{\rm Sym}^{r-1}\bigl(\widehat{\sigma}^1_1\bigr) =
{\rm Id}_{{\rm Sym}^{r-1}(J^1(\mathcal L))}$ whenever $r$ is an odd integer.

Recall that ${\mathcal C}_X(r)$ denotes the space of all holomorphic connections on ${\rm
Sym}^{r-1}\bigl(J^1(\mathcal L)\bigr)$ such that their induced connection on $\det {\rm Sym}^{r-1}\bigl(J^1(\mathcal
L)\bigr)$ is the trivial connection on ${\mathcal O}_X$. Let
\begin{equation}\label{f6}
\gamma\colon\ {\mathcal C}_X(r) \longrightarrow {\mathcal C}_X(r)
\end{equation}
be the anti-holomorphic map that sends any holomorphic connection $D$ to the
holomorphic connection on ${\rm
Sym}^{r-1}\bigl(J^1(\mathcal L)\bigr)$ given by the holomorphic connection ${\rm
Sym}^{r-1}\bigl(\widehat{\sigma}^1_1\bigr)^*\sigma^*_X\overline{D}$ through the
isomorphism in \eqref{f5}, where $\sigma^*_X\overline{D}$ is the connection on
$\sigma^*_X\overline{J^1(\mathcal L)}$ induced by $D$. It can be shown that $\gamma$ in \eqref{f6} is an
involution. Indeed, the automorphism of ${\rm Sym}^{r-1}\bigl(J^1(\mathcal L)\bigr)$ given by multiplication by $-1$ preserves any
connection on ${\rm Sym}^{r-1}\bigl(J^1(\mathcal L)\bigr)$. So even if the holomorphic line bundle
$\mathcal L$ is quaternionic, the map $\gamma$ remains to be of order two.

The group ${\mathbb Z}/2{\mathbb Z}$ acts on ${\mathcal C}_X(r)$ via the involution $\gamma$
in \eqref{f6}, and the group of automorphisms $\text{Aut}\bigl({\rm Sym}^{r-1}\bigl(J^1(\mathcal L)\bigr)\bigr)$ also acts on
${\mathcal C}_X(r)$ through \eqref{f2}. On the other hand, there is an anti-holomorphic group
automorphism of order two
\begin{gather}\label{f8}
\tau_r\colon\ \text{Aut}\bigl({\rm Sym}^{r-1}\bigl(J^1(\mathcal L)\bigr)\bigr) \longrightarrow
\text{Aut}\bigl({\rm Sym}^{r-1}\bigl(J^1(\mathcal L)\bigr)\bigr)
\end{gather}
that sends any $\beta \in \text{Aut}\bigl({\rm Sym}^{r-1}\bigl(J^1(\mathcal L)\bigr)\bigr)$ to
\[
\bigl({\rm Sym}^{r-1}\bigl(\widehat{\sigma}^1_1\bigr)\bigr)^{-1}\circ\bigl(\sigma^*_X\overline{\beta}\bigr)\circ
{\rm Sym}^{r-1}\bigl(\widehat{\sigma}^1_1\bigr) \in
\text{Aut}\bigl({\rm Sym}^{r-1}(J^1(\mathcal L))\bigr),
\]
where $\sigma^*_X\overline{\beta} \in \text{Aut}\bigl(\sigma^*_X\overline{{\rm Sym}^{r-1}
\bigl(J^1(\mathcal L)\bigr)}\bigr)$ is the automorphism given by $\beta$.

\begin{Lemma}\label{lem3}
The action of ${\mathbb Z}/2{\mathbb Z}$ on ${\mathcal C}_X(r)$ given by the involution $\gamma$
in \eqref{f6} descends to an involution
\begin{equation}\label{f7}
\beta\colon\ {\rm OSL}_X(r) \longrightarrow {\rm OSL}_X(r)
\end{equation}
of ${\mathbb Z}/2{\mathbb Z}$ on the quotient space of ${\rm SL}(r,\mathbb{C})$-opers
${\rm OSL}_X(r)$ $($see Definition {\rm \ref{deo}}$)$.
\end{Lemma}

\begin{proof}
The lemma follows from the above construction of $\gamma$ and from noticing that the
actions of ${\mathbb Z}/2{\mathbb Z}$ and ${\rm Aut}\bigl({\rm Sym}^{r-1}\bigl(J^1(\mathcal L)\bigr)\bigr)$
on ${\mathcal C}_X(r)$ are related as follows:
\[
\gamma\circ\psi (D) = \tau_r(\psi)\circ\gamma(D)
\]
for all $\psi \in {\rm Aut}\bigl({\rm Sym}^{r-1}\bigl(J^1(\mathcal L)\bigr)\bigr)$ and $D \in {\mathcal C}_X(r)$, where $\tau_r$
and $\gamma$ are constructed in \eqref{f8} and~\eqref{f6} respectively.
\end{proof}

\begin{Remark} Since $\gamma$ in \eqref{f6} is anti-holomorphic, it
follows immediately that $\beta$ in \eqref{f7} is also anti-holomorphic.
\end{Remark}

\begin{Definition}\label{de1}
An ${\rm SL}(r,{\mathbb C})$-{\it oper} on $X$ is in the real slice of ${\rm SL}(r,\mathbb{C})$-opers if it is a fixed point of the
involution $\beta$ in \eqref{f7} of Lemma \ref{lem3}.
\end{Definition}

In order to study the slice of real ${\rm SL}(r,{\mathbb C})$-{\it opers} on $X$, we shall construct a conjugate linear involution of $H^0\bigl(X, {\rm Diff}^r_X\bigl({\mathcal L}^{r-1},
{\mathcal L}^{-r-1}\bigr)\bigr)$.

\begin{Lemma}\label{bigb} There is a natural
conjugate linear involution \begin{equation}\label{f9}
{\mathcal B} \colon\ H^0\bigl(X, {\rm Diff}^r_X\bigl({\mathcal L}^{r-1}, {\mathcal L}^{-r-1}\bigr)\bigr) \longrightarrow
H^0\bigl(X, {\rm Diff}^r_X\bigl({\mathcal L}^{r-1}, {\mathcal L}^{-r-1}\bigr)\bigr)
\end{equation}
induced by the anti-holomorphic involution $\sigma$ of $X$.
\end{Lemma}

\begin{proof} Given any $
\delta \in H^0\bigl(X, {\rm Diff}^r_X\bigl({\mathcal L}^{r-1},
{\mathcal L}^{-r-1}\bigr)\bigr),
$
the holomorphic differential operator on $\sigma^*_X \overline{\mathcal L}^{r-1}$ defined by
$
\sigma^*\overline{\delta}\colon J^r\bigl(\sigma^*_X \overline{\mathcal L}^{r-1}\bigr)
 \longrightarrow \sigma^*_X \overline{\mathcal L}^{-r-1}
$
sends any locally defined holomorphic section $s$ of ${\mathcal L}^{r-1}$ to
$\sigma^*\overline{\delta (s)}$; recall that the holomorphic sections of
$\sigma^*_X \overline{\mathcal L}^{r-1}$ over an open subset~$U \subset X$ are identified with the
holomorphic sections of ${\mathcal L}^{r-1}$ over $\sigma_X(U)$.
Consider the diagram
\[
\begin{matrix}
J^r\bigl({\mathcal L}^{r-1}\bigr) &\xrightarrow{\widehat{\sigma}^r_{r-1}} &
J^r\bigl(\sigma^*_X \overline{\mathcal L}^{r-1}\bigr)\\
 \Big\downarrow \delta && \Big\downarrow \sigma^*\overline{\delta}\\
{\mathcal L}^{-r-1} &\xrightarrow{\widehat{\sigma}^0_{-r-1}}&
\sigma^*_X \overline{\mathcal L}^{-r-1},
\end{matrix}
\]
where $\widehat{\sigma}^r_{r-1}$ and $\widehat{\sigma}^0_{-r-1}$ are the holomorphic isomorphisms
in \eqref{s2}, which needs not be commutative. This diagram shows that
\[
\bigl(\widehat{\sigma}^0_{-r-1}\bigr)^{-1}\circ \bigl(\sigma^*\overline{\delta}\bigr) \circ\widehat{\sigma}^r_{r-1}
 \in H^0\bigl(X, {\rm Diff}^r_X\bigl({\mathcal L}^{r-1}, {\mathcal L}^{-r-1}\bigr)\bigr).
\]

The map ${\mathcal B}$ from \eqref{f9} can then be defined as
the map that sends any $\delta$ to the differential operator
$\bigl(\widehat{\sigma}^0_{-r-1}\bigr)^{-1}\circ \bigl(\sigma^*\overline{\delta}\bigr)
\circ\widehat{\sigma}^r_{r-1}$ constructed above from it.
Finally, it should be noted that ${\mathcal B}$ is an
involution even if the holomorphic line bundle $\mathcal L$ is quaternionic.
\end{proof}

\begin{Theorem}\label{prop2}
The involution $\mathcal B$ in \eqref{f9} preserves the subset ${\rm OSL}_X(r)$ in Proposition {\rm \ref{f3}}.
The restriction of the map $\mathcal B$ to this subset ${\rm OSL}_X(r)$ coincides with the
involution $\beta$ in \eqref{f7}.
\end{Theorem}

\begin{proof}
Recall from Proposition \ref{f3} that ${\rm OSL}_X(r)$ is the locus of all
\[
\delta \in H^0\bigl(X, {\rm
Diff}^r_X\bigl({\mathcal L}^{r-1}, {\mathcal L}^{-r-1}\bigr)\bigr)\]
 such that the symbol of $\delta$ is
the constant function $1$ and the sub-leading term of $\delta$ vanishes identically.
Therefore, it is evident that $\mathcal B$ preserves the subset ${\rm OSL}_X(r)$.
Given a holomorphic connection $D$ on
${\rm Sym}^{r-1}\bigl(J^1(\mathcal L)\bigr)$ such that the induced connection on $\det {\rm
Sym}^{r-1}\bigl(J^1(\mathcal L)\bigr)$ is the trivial connection on ${\mathcal O}_X$,
through the map \smash{$\widetilde{\psi}$} in \eqref{psn} we have
\begin{equation}\label{a1}
\widetilde{\psi}\circ{\rm Sym}^{r-1}\bigl(\widehat{\sigma}^1_1\bigr) = \widehat{\sigma}^{r-1}_{r-1}
\circ\widetilde{\psi}
\end{equation}
(see \eqref{f5} and \eqref{s2} for ${\rm Sym}^{r-1}\bigl(\widehat{\sigma}^1_1\bigr)$ and
$\widehat{\sigma}^{r-1}_{r-1}$, respectively). Moreover, from \eqref{a2}
there is a~unique differential operator $\delta$ such that
\begin{itemize}\itemsep=0pt
\item the symbol of $\delta$ is the constant function $1$ on $X$, and

\item the flat vector bundle on $X$ given by the sheaf of solutions of $\delta$ is isomorphic to
$J^{r-1}\!\bigl(\!{\mathcal L}^{\otimes (\!r-1\!)}\bigr)$ equipped with the above connection \smash{$\bigl(\widetilde{\psi}^{-1}\bigr)^*D$}.
\end{itemize}
Since $\delta$ is the image of $D$ under the map in \eqref{z1b}, the second
part of the theorem follows from~\eqref{a1}.
\end{proof}

Theorem \ref{prop2} has the following immediate consequence, which characterizes the real slice of ${\rm SL}(r,\mathbb{C})$-opers defined above.

\begin{Proposition}\label{corpr2}
An element of the moduli space of ${\rm SL}(r,\mathbb{C})$-opers
\[
\delta \in {\rm OSL}_X(r) \subset H^0\bigl(X,
{\rm Diff}^r_X\bigl({\mathcal L}^{\otimes (r-1)}, {\mathcal L}^{-r-1}\bigr)\bigr)
\]
is in the real slice of ${\rm SL}(r,{\mathbb C})$-opers if and only if
${\mathcal B}(\delta) = \delta$, where $\mathcal B$ is the involution defined in~\eqref{f9}.
\end{Proposition}

\section{Another description of the involution}

In order to study the real slice of ${\rm SL}(r,\mathbb{C})$-opers defined through Proposition \ref{corpr2} in the previous section, we shall dedicate this section to describing the involution from a different perspective.

\subsection{Anti-holomorphic involutions on differential operators} Consider as in previous sections (see \eqref{z1}) the projections
\[
p_i\colon\ X\times X \longrightarrow X.
\]
Then, given two holomorphic vector bundles $A$ and $B$ on $X$, we shall denote the holomorphic vector bundle $(p^*_1 A)\otimes (p^*_2 B)$ on
$X\times X$ by $A\boxtimes B$.
These holomorphic vector bundles produce the natural vector bundles
\begin{eqnarray}\label{abr}
\mathcal{J} := A\boxtimes (B^*\otimes K_X)\qquad \text{ and }\qquad
\mathcal{J}_d := A\boxtimes (B^*\otimes K_X)\otimes {\mathcal O}_{X\times X}((d+1){\Delta})\end{eqnarray} on $X\times X$,
where $\Delta$ is the diagonal divisor defined in \eqref{z2}. These vector bundles fit in a short exact sequence of coherent
sheaves on $X\times X$ given by
\begin{gather}
0 \longrightarrow \mathcal{J} \longrightarrow \mathcal{J}_d \longrightarrow {\mathcal Q}_d(A, B)
 := \frac{\mathcal{J}_d}{\mathcal{J}} \longrightarrow 0,
\label{k1}
\end{gather}
where the support of ${\mathcal Q}_d(A, B)$ in \eqref{k1}
is the non-reduced divisor $(d+1){\Delta}$. The direct image
\begin{equation}\label{k2}
{\mathcal K}_d(A, B) := p_{1*} {\mathcal Q}_d(A, B)
\end{equation}
is a holomorphic vector bundle on $X$, and from \cite[Section 2.1]{BS}, \cite[p.~25, equation~(5.1)]{Bi},
\cite[Section 3.1, p.~1314]{BB}) one has that
\begin{equation}\label{k3}
{\mathcal K}_d(A, B) = \operatorname{Hom}\bigl(J^d(B), A\bigr) = \operatorname{Diff}^d_X(B, A).
\end{equation}

For $d \geq 1$, the sheaf ${\mathcal Q}_d(A, B)$ in \eqref{k1} fits in the following short
exact sequence of sheaves on $X\times X$, where the abbreviations of \eqref{abr} are used,
\begin{gather}\label{ez1}
0 \longrightarrow {\mathcal Q}_{d-1}(A, B) \longrightarrow {\mathcal Q}_{d}(A, B) \longrightarrow
 \frac{\mathcal{J}_d}{\mathcal{J}_{d-1}} \longrightarrow 0.
\end{gather}
Note that the above sheaf $\mathcal{J}_d/\mathcal{J}_{d-1}$ is supported on the reduced divisor ${\Delta}$.
Taking direct image of the short exact sequence in \eqref{ez1} by the projection~$p_1$, we get the following
short exact sequence of holomorphic vector bundles on $X$
\begin{gather}
0 \longrightarrow p_{1*} {\mathcal Q}_{d-1}(A, B) \longrightarrow p_{1*} {\mathcal Q}_d(A, B)
 = {\mathcal K}_d(A, B)
\stackrel{p_0}{\longrightarrow} p_{1*} \left( \frac{\mathcal{J}_d}{\mathcal{J}_{d-1}}\right) \longrightarrow 0.\label{k4}
\end{gather}

From Poincar\'e adjunction formula, one has that ${\mathcal O}_{X\times X}({\Delta})\big\vert_{\Delta}$
is the normal bundle of ${\Delta}$ \cite[p.~146]{GH}, and therefore ${\mathcal O}_{X\times X}({\Delta})\big\vert_{\Delta} = TX$, using the identification of
${\Delta}$ with $X$ defined by~$x \longmapsto (x, x)$. Hence, we have
\[
p_{1*}\left( \frac{\mathcal{J}_d}{\mathcal{J}_{d-1}}\right) = \operatorname{Hom}(B, A)\otimes (TX)^{\otimes d}.
\]

\begin{Remark} The isomorphism in \eqref{k3} and the projection $p_0$ in \eqref{k4} together produce a~homomorphism
\begin{equation}\label{k5}
\operatorname{Diff}^d_X(B, A) \longrightarrow \operatorname{Hom}(B, A)\otimes (TX)^{\otimes d},
\end{equation}
which is the symbol map on differential operators constructed in \eqref{es}.\end{Remark}

Recall from Section \ref{invo} that $\mathcal L := \bigl(K^{1/2}_X\bigr)^*$, and consider now an element of the moduli space of
${\rm SL}(r, {\mathbb C})$-opers on $X$ given by
\begin{equation}\label{d1}
\delta \in {\rm OSL}_X(r) \subset
H^0\bigl(X, {\rm Diff}^r_X\bigl({\mathcal L}^{\otimes (r-1)}, {\mathcal L}^{-r-1}\bigr)\bigr).
\end{equation}
From \eqref{k2} and \eqref{k3}, we know that $\delta$ corresponds to a section
\begin{equation}\label{sd}
\widetilde{S}_\delta \in H^0((r+1)\Delta, \bigl({\mathcal L}^{-r-1}\boxtimes {\mathcal L}^{-r-1}\bigr)
{\mathcal O}_{X\times X}((r+1)\Delta)).
\end{equation}
Note that the restriction of
\[
\bigl({\mathcal L}^{-r-1}\boxtimes {\mathcal L}^{-r-1}\bigr){\mathcal O}_{X\times X}((r+1)\Delta)
 \longrightarrow X\times X
\]
to $\Delta \subset X\times X$ is ${\mathcal O}_\Delta$, and that $\mathcal{J}_r =
\bigl({\mathcal L}^{-r-1}\boxtimes {\mathcal L}^{-r-1}\bigr){\mathcal O}_{X\times X}((r+1)\Delta)$ for the vector
bundles in \eqref{abr} defined as
\[A = {\mathcal L}^{-r-1}\qquad \text{ and }\qquad B = {\mathcal L}^{r-1}.\]

\begin{Lemma}
There is a natural section
\[
S_\delta \in H^0(3\Delta, \mathcal{J}_0) = H^0\bigl(3\Delta,
\bigl({\mathcal L}^{-1}\boxtimes {\mathcal L}^{-1}\bigr){\mathcal O}_{X\times X}(\Delta)\bigr)
\]
such that
\begin{itemize}\itemsep=0pt
\item $\widetilde{S}_\delta\big\vert_{3\Delta} = (S_\delta)^{\otimes (r+1)}$,

\item the restriction of $S_\delta$ to $\Delta \subset X\times X$ coincides with the section of
${\mathcal O}_\Delta$ given by the constant function $1$, and

\item $S_\delta\big\vert_{2\Delta}$ is anti-invariant under the involution of $2\Delta$ obtained by restricting
the involution of $X\times X$ defined by $(x_1, x_2) \longmapsto (x_2, x_1)$, in other words, this
involution takes $S_\delta\big\vert_{2\Delta}$ to~$-S_\delta\big\vert_{2\Delta}$.
\end{itemize}
\end{Lemma}

\begin{proof} This follows from the above analysis, by noting that the given condition that the symbol of
$\delta$ is the constant function $1$ is equivalent to the statement that the restriction of the section~$\widetilde{S}_\delta$ (constructed in \eqref{sd}) to $\Delta \subset X\times X$ is the section of
${\mathcal O}_\Delta$ given by the constant function $1$ (see~\eqref{k5}). The given condition that
the sub-leading term of $\delta$ vanishes identically is equivalent to the condition that the section
$\widetilde{S}_\delta\big\vert_{2\Delta}$ is anti-invariant under the involution of $2\Delta$ obtained by restricting
the involution of $X\times X$ defined by $(x_1, x_2) \longmapsto (x_2, x_1)$.
\end{proof}

We will show that the section of $\mathcal{J}_r$ for $r = 2$ given by
\begin{equation}\label{s1}
(S_\delta)^{\otimes 3} \in H^0(3\Delta, \mathcal{J}_2)
\end{equation}
naturally defines a projective structure on the Riemann surface $X$.

\begin{Proposition}
The section $(S_\delta)^{\otimes 3}$ in \eqref{s1} naturally defines a projective structure on
the Riemann surface $X$.
\end{Proposition}

\begin{proof}From \eqref{k2} and \eqref{k3}, we know that $(S_\delta)^{\otimes 3}$ in \eqref{s1} gives a differential operator
\begin{equation}\label{s3}
\delta_0 \in H^0\bigl(X, {\rm Diff}^2_X\bigl({\mathcal L}, {\mathcal L}^{-3}\bigr)\bigr).
\end{equation}
Since the restriction of $S_\delta$ to $\Delta \subset X\times X$ coincides with the section of
${\mathcal O}_\Delta$ given by the constant function $1$, we conclude that the symbol of the differential
operator $\delta_0$ is the constant function~$1$ (see \eqref{k5}). Since
$S_\delta\big\vert_{2\Delta}$ is anti-invariant under the involution of $2\Delta$ obtained by restricting
the involution of $X\times X$ defined by $(x_1, x_2) \longmapsto (x_2, x_1)$, we conclude that
the sub-leading term of the differential operator $\delta_0$ vanishes identically.
Therefore, we have
\[
\delta_0 \in {\rm OSL}_X(2)
\]
(see Proposition \ref{f3}). Since the space ${\rm OSL}_X(2)$ is identified with the space of projective structures on $X$,
the differential operator $\delta_0$ in \eqref{s3} defines a projective structure on
the Riemann surface $X$.
\end{proof}

Consider the flat connection $D_0$ given by the sheaf of solutions of the differential operator $\delta_0$ in~\eqref{s3}.
The underlying holomorphic vector bundle for it is $J^1(\mathcal L)$.
The holomorphic connection on $\text{Sym}^{r-1}\bigl(J^1(\mathcal L)\bigr)$ induced by $D_0$ will be denoted by
$\text{Sym}^{r-1}(D_0)$, and from \eqref{psn} it produces a holomorphic isomorphism
\[\text{Sym}^{r-1}\bigl(J^1(\mathcal L)\bigr) \longrightarrow J^{r-1}\bigl({\mathcal L}^{r-1}\bigr).
\]
The holomorphic connection on $J^{r-1}\bigl({\mathcal L}^{r-1}\bigr)$
given by $\text{Sym}^{r-1}(D_0)$ using this isomorphism will be denoted by $\text{Sym}^{r-1}(D_0)'$.

For $D$ denote the holomorphic connection on $J^{r-1}\bigl({\mathcal L}^{r-1}\bigr)$ given by the sheaf of solutions of
the differential operator $\delta$ in \eqref{d1}, we shall define
\[
\theta_\delta := D- \text{Sym}^{r-1}(D_0)' \in H^0\bigl(X, \text{End}\bigl(J^{r-1}\bigl({\mathcal L}^{r-1}\bigr)\bigr)
\otimes K_X\bigr).
\]
Taking the trace for all $1 \leq i \leq r$ given by
\begin{equation}\label{d3}
\text{trace}\bigl((\theta_\delta)^i\bigr) \in H^0\bigl(X, K^{\otimes i}_X\bigr),
\end{equation}
note that $\text{trace}(\theta_\delta) = 0 = \text{trace}\bigl((\theta_\delta)^2\bigr)$.
\subsection{Real slices of fibrations}
In search of an interesting parametrization of real slices of ${\rm SL}(r,\mathbb{C})$-opers similar to those one can define for Higgs bundles though Hitchin's integrable system \cite{Hi}, consider the map
\begin{equation}\label{d2}
\Psi\colon\ {\rm OSL}_X(r) \longrightarrow {\rm OSL}_X(2)\times \left(\bigoplus_{i=3}^r H^0\bigl(X, K^{\otimes i}_X\bigr)\right)
\end{equation}
 that sends any differential operator $\delta$ as defined in \eqref{f3} to
$\bigl(\delta_0, \bigoplus_{i=3}^r \text{trace}\bigl((\theta_\delta)^i\bigr)\bigr)$, where ${\delta_0 \in {\rm OSL}_X(2)}$
is constructed in \eqref{s3} and the traces $\text{trace}\bigl((\theta_\delta)^i\bigr)$ are the sections in \eqref{d3}.

Consider an anti-holomorphic involution $\sigma_X$ of $X$ as in Section \ref{realoper}, and the induced involution~$F$
from Lemma \ref{lemmae8}. From equation \eqref{ec}, it follows that $F$ restricts to an anti-holomorphic involution
\[
F_X\colon\ {\rm OSL}_X(2) \longrightarrow {\rm OSL}_X(2).
\]
From Theorem \ref{prop2}, in the case when $r = 2$ this $F_X$ coincides with $\mathcal B$
constructed in \eqref{f9}, and~$\beta$ constructed in \eqref{f7}. On the other hand, the differential
${\rm d}\sigma\colon T^{\mathbb R}X \longrightarrow \sigma^*T^{\mathbb R}X$ produces a~holomorphic involution $TX \longrightarrow \sigma^*\overline{TX}$. Through these maps, one can define an anti-holomorphic
linear automorphism of $\omega \longmapsto \sigma^*_X\overline{\omega}$ of $H^0\bigl(X, K^{\otimes i}_X\bigr)$.
Then there is a natural involution
\begin{align}
{\mathcal A}_1\colon\ {\rm OSL}_X(2)\times \left(\bigoplus_{i=3}^r H^0\bigl(X, K^{\otimes i}_X\bigr)\right)&
\longrightarrow {\rm OSL}_X(2)\times \left(\bigoplus_{i=3}^r H^0\bigl(X, K^{\otimes i}_X\bigr)\right),\nonumber\\
\left(\eta, \bigoplus_{i=3}^r \omega_i\right)&\longmapsto \left(F_X(\eta),
\bigoplus_{i=3}^r \sigma^*_X\overline{\omega}_i\right). \label{d4}
\end{align}
Through the above involution, Theorem \ref{prop2} leads to the following statement.

\begin{Proposition}\label{lem4}
Let ${\mathcal A} = \Psi^{-1}\circ{\mathcal A}_1\circ\Psi$ be the involution of
${\rm OSL}_X(r)$, where ${\mathcal A}_1$ and $\Psi$
are constructed in \eqref{d4} and \eqref{d2} respectively. Then ${\mathcal A}$ coincides
with the involution $\beta$ in \eqref{f7}. Also ${\mathcal A}$ coincides
with the restriction of $\mathcal B$ $($in \eqref{f9}$)$ to ${\rm OSL}_X(r)$.
\end{Proposition}

Finally, Proposition \ref{lem4} has the following immediate consequence.

\begin{Corollary}\label{corpr3}
An element
\[
\left(\eta, \bigoplus_{i=3}^r \omega_i\right) \in
{\rm OSL}_X(2)\times \left(\bigoplus_{i=3}^r H^0\bigl(X, K^{\otimes i}_X\bigr)\right)
\]
is within the real slice of ${\rm SL}(r,{\mathbb C})$-opers if and only if
$F_X(\eta) = \eta$ and $\omega_i = \sigma^*_X\overline{\omega}_i$ for all $3 \leq i \leq r$.
\end{Corollary}

\section{Nondegenerate immersions into a projective space}

Let $X$ be a connected Riemann surface and consider a holomorphic immersion
\[
\phi\colon\ X \longrightarrow {\mathbb C}{\mathbb P}^d.
\]
Given a point $x \in X$ and a hyperplane
$H \subset {\mathbb C}{\mathbb P}^d$ such that $\phi(x) \in H$,
we say that the order of contact of $X$ with $H$ at $x$ is at least $n$
if the following condition holds: Every holomorphic function $h$ defined on some
neighborhood $U$ of $\phi(x) \in {\mathbb C}{\mathbb P}^d$ such that
$h\big\vert_{U\cap H} = 0$, the order of vanishing of $h\circ\phi$ at
$x$ is at least $n$.

\begin{Definition}The {\it order of contact} of $X$ with $H$ at $x$ is said to
be $n$ if
\begin{itemize}\itemsep=0pt
\item the order of contact of $X$ with $H$ at $x$ is at least $n$, and

\item the order of contact of $X$ with $H$ at $x$ is not at least $n+1$.
\end{itemize}
\end{Definition}
In other words, the order of contact of $X$ with $H$ at $x$ is $n$ if
\begin{itemize}\itemsep=0pt
\item the order of contact of $X$ with $H$ at $x$ is at least $n$, and

\item there is a holomorphic function $h$ defined on some
neighborhood $U$ of $\phi(x) \in {\mathbb C}{\mathbb P}^d$ such that
$h\big\vert_{U\cap H} = 0$ and the order of vanishing of $h\circ\phi$ at
$x$ is exactly $n$.
\end{itemize}

For any point $x \in X$, there is a hyperplane $H^x \subset {\mathbb C}{\mathbb
P}^d$ such that the order of contact of $X$ with $H^x$ at $x$ is at least $d-1$.

\begin{Definition}We say that the map $\phi$ is \textit{nondegenerate at $x$} if the order of contact of $X$,
at $x$, with every hyperplane $H \subset {\mathbb C}{\mathbb P}^d$ containing $\phi(x)$
is at most $d-1$. Equivalently, $\phi$ is nondegenerate at $x$ if there is
no hyperplane $H \subset {\mathbb C}{\mathbb P}^d$ such that the
order of contact of $X$ with~$H$ at $x$ is at least $d$.
The map $\phi$ is said to be \textit{nondegenerate} if it is nondegenerate at every~$x \in X$.
\end{Definition}

Consider now $X$ to be a {\it compact} connected Riemann surface and fix a point $x_0 \in X$.
Given the universal cover
$
\varpi\colon \widetilde{X} \longrightarrow X
$
 one has that $\pi_1(X, x_0)$ acts on
$\widetilde{X}$ as deck transformations. Consider all pairs of the form
$(\rho, \phi)$, where
$
\rho\colon \pi_1(X, x_0) \longrightarrow \text{SL}(r, {\mathbb C})
$
is a homomorphism and
\[
\phi\colon\ \widetilde{X} \longrightarrow {\mathbb C}{\mathbb P}^{r-1}
\]
is a holomorphic immersion such that
\begin{itemize}\itemsep=0pt
\item[(1)] $\phi$ is nondegenerate, and

\item[(2)] $\phi\circ\gamma = \rho(\gamma)\circ\phi$ for all $\gamma \in
\pi_1(X, x_0)$ (the standard action of $\text{SL}(r, {\mathbb C})$
on ${\mathbb C}{\mathbb P}^{r-1}$ is being used here).
\end{itemize}

\begin{Definition}\label{space}Two pairs $(\rho, \phi)$ and $(\rho_1, \phi_1)$ as above are
 called \textit{equivalent} if there is an element~$G \in
\text{SL}(r, {\mathbb C})$ such that
\begin{itemize}\itemsep=0pt
\item[(1)] $\rho_1(\gamma) = G^{-1}\rho(\gamma) G$ for all
$\gamma \in \pi_1(X, x_0)$, and

\item[(2)] $\phi_1 = G^{-1}\circ\phi$.
\end{itemize}
\end{Definition}

For any given $r \geq 2$, by ${\mathcal M}(r)$ we shall denote the space of all equivalence
classes of pairs~$(\rho, \phi)$ of the above type.
Recall that given another base point
$x'_0 \in X$, the fundamental group $\pi_1(X, x'_0)$ is identified with
$\pi_1(X, x_0)$ uniquely up to an inner automorphism. More precisely, a choice of
a homotopy class of paths between $x_0$ and $x'_0$ identifies $\pi_1(X, x_0)$
with $\pi_1(X, x'_0)$ and it also identifies $\widetilde{X}$ with the universal cover of $X$
corresponding to $x_0'$. From these it follows immediately
that ${\mathcal M}(r)$ does not depend on the choice of the base point $x_0$.
Moreover, it is known that ${\mathcal M}(r)$ is identified with the space of all ${\rm SL}(r,
{\mathbb C})$-opers on $X$. This identification is briefly described in the proof
of Proposition \ref{prop3}.

\subsection{Real structures} Consider as in previous sections an anti-holomorphic involution
\[
\sigma\colon\ X \longrightarrow X,
\]
and denote by
\begin{equation}\label{o0}
\sigma_*\colon\ \pi_1(X, x_0) \longrightarrow \pi_1(X, \sigma(x_0))
\end{equation}
 the homomorphism induced by $\sigma$. Using the homomorphism $(\sigma_*)^{-1}$, the
action of $\pi_1(X, x_0)$ on~$\widetilde{X}$ produces an action of
$\pi_1(X, \sigma(x_0))$ on $\widetilde{X}$. Let
\[
\varpi'\colon\ \widetilde{X}' \longrightarrow X
\]
be the universal covering corresponding to the point $\sigma(x_0)$. Note whilst that there
is no natural map between the universal covers $\widetilde{X}'$ and $\widetilde{X}$,
the self-map $\sigma$ of $X$ has a natural lift to the universal covering space
\begin{equation}\label{o-1}
\widetilde{\sigma}\colon\ \widetilde{X} \longrightarrow \widetilde{X}'.
\end{equation}
\begin{Remark}From the above constructions, one has the following:
\begin{itemize}\itemsep=0pt
\item[(1)] $\varpi'\circ \widetilde{\sigma} = \varpi$,

\item[(2)] $\widetilde{\sigma}$ takes the base point in $\varpi^{-1}(x_0)$ to the base point
in $(\varpi')^{-1}(\sigma(x_0))$,

\item[(3)] $\widetilde{\sigma}$ is an anti-holomorphic diffeomorphism, and

\item[(4)] $\widetilde{\sigma}$ is $\pi_1(X, \sigma(x_0))$-equivariant \big(recall that
$\pi_1(X, \sigma(x_0))$ acts on $\widetilde{X}$\big).
\end{itemize}
\end{Remark}
\begin{Lemma}\label{newinvo} The involution $\sigma$ on $X$ induces an involution $\Gamma$ on ${\mathcal M}(r)$ of all equivalence
classes of pairs $(\rho, \phi)$.
\end{Lemma}
\begin{proof}
In what follows we shall describe an involution on the space ${\mathcal M}(r)$ of all equivalence
classes of pairs $(\rho, \phi)$ constructed through Definition \ref{space}.
Given a pair
\[
(\rho, \phi) \in {\mathcal M}(r),
\]
one can define a homomorphism
\begin{align}
\widehat{\rho}\colon\ \pi_1(X, \sigma(x_0)) &\longrightarrow \text{SL}(r, {\mathbb C}),\nonumber\\
\gamma &\mapsto \overline{\rho((\sigma_*)^{-1}(\gamma))},\label{o2}
\end{align}
where $\sigma_*$ is the homomorphism in
\eqref{o0}.
Moreover, one can also define the map
\begin{align}
\widehat{\phi}\colon\ \widetilde{X}' &\longrightarrow {\mathbb C}{\mathbb P}^{r-1},\nonumber\\
y&\mapsto\overline{\phi\bigl(\widetilde{\sigma}^{-1}(y)\bigr)}, \label{o3}
\end{align}
where $\widetilde{\sigma}$ is the map in \eqref{o-1}. Now it is straightforward to
check that the pair $\bigl(\widehat{\rho}, \widehat{\phi}\bigr)$ constructed in \eqref{o2}
and \eqref{o3} defines an element of ${\mathcal M}(r)$. Let
\begin{equation}\label{o4}
\Gamma\colon\ {\mathcal M}(r) \longrightarrow {\mathcal M}(r)
\end{equation}
be the map that sends any pair $(\rho, \phi)$ to $\bigl(\widehat{\rho}, \widehat{\phi}\bigr)$
constructed from it in \eqref{o2} and \eqref{o3}. It is evident that $\Gamma$ is an involution.\end{proof}

As mentioned before,
${\mathcal M}(r)$ is identified with the space of all
${\rm SL}(r, {\mathbb C})$-opers on $X$. Hence, the involution $\Gamma$ constructed in Lemma \ref{newinvo} via \eqref{o4}
is an involution of ${\rm OSL}_X(r)$ (see Definition~\ref{deo}) and it can be understood in terms of the involutions studied in previous sections.

\begin{Proposition}\label{prop3}
The involution $\Gamma$ of ${\rm OSL}_X(r)$ coincides with the involution
$\beta$ constructed in Lemma {\rm \ref{lem3}}.
\end{Proposition}

\begin{proof}
This is deduced by examining the identification between ${\rm OSL}_X(r)$ and
${\mathcal M}(r)$. The map ${\mathcal M}(r) \longrightarrow {\rm OSL}_X(r)$
is constructed as follows. Take any pair $(\rho, \phi)$ giving an element of
${\mathcal M}(r)$. Let
\[
E_0 = {\mathbb C}{\mathbb P}^{r-1}\times {\mathbb C}^r \longrightarrow
{\mathbb C}{\mathbb P}^{r-1}
\]
be the trivial holomorphic vector bundle equipped with the unique holomorphic
connection $D_0$ on $E_0$ (it is the trivial connection). The vector bundle
$E_0$ has a tautological holomorphic line subbundle
\begin{equation}\label{l1}
{\mathcal L} = {\mathcal O}_{{\mathbb C}{\mathbb P}^{r-1}}(-1) \subset E_0,
\end{equation}
whose fiber over any $z \in {\mathbb C}{\mathbb P}^{r-1}$ is the line in
${\mathbb C}^r$ represented by $z$.

The homomorphism $\rho\colon
\pi_1(X, x_0) \longrightarrow \text{SL}(r, {\mathbb C})$ and the standard
action of $\text{SL}(r, {\mathbb C})$ on ${\mathbb C}{\mathbb P}^{r-1}$ together
produce an action of $\pi_1(X, x_0)$ on ${\mathbb C}{\mathbb P}^{r-1}$. Consider
the diagonal action of $\pi_1(X, x_0)$ on the pulled back vector bundle
\[
\phi^*E_0 = \widetilde{X}\times {\mathbb C}^r \longrightarrow \widetilde{X}.
\]
This action of $\pi_1\!(X, x_0)$ on $\phi^*E_0$ preserves the holomorphic
line subbundle $\phi^*{\mathcal L} \!\subset \!\phi^*E_0$ (see~\eqref{l1}). Moreover,
the action of $\pi_1(X, x_0)$ on $\phi^*E_0$ preserves the pulled back
connection $\phi^* D_0$. Therefore, one can define the vector bundle
\begin{equation}\label{l2}
{\mathcal V} := (\phi^*E_0)/\pi_1(X, x_0) \longrightarrow
\widetilde{X}/\pi_1(X, x_0) = X,
\end{equation}
which is a holomorphic vector bundle equipped with
\begin{itemize}\itemsep=0pt
\item a holomorphic connection $D$ induced by $\phi^* D_0$, and

\item a holomorphic line subbundle
\[
\widetilde{\mathcal L} := (\phi^*{\mathcal L})/\pi_1(X, x_0) \subset
(\phi^*E_0)/\pi_1(X, x_0) = {\mathcal V}.
\]
\end{itemize}
The holomorphic section of the line bundle
$\bigwedge^r (\phi^*E_0) = \widetilde{X}\times{\mathbb C} \longrightarrow
\widetilde{X}$ given by the constant function $1$ on $\widetilde{X}$ is preserved
by the action of $\pi_1(X, x_0)$ on $\bigwedge^r (\phi^*E_0)$ induced by the action of
$\pi_1(X, x_0)$ on $\phi^*E_0$. Furthermore, this constant section is preserved by the
connection on~$\bigwedge^r (\phi^*E_0)$ induced by the connection
$\phi^* D_0$ on $\phi^*E_0$. Consequently, the vector bundle ${\mathcal V}$ in~\eqref{l2} defines a
principal $\text{SL}(r, {\mathbb C})$-bundle on $X$, and $D$ is
a holomorphic connection on this principal $\text{SL}(r, {\mathbb C})$-bundle. From the
non-degeneracy condition of the map $\phi$ it can be deduced that the
line subbundle $\widetilde{\mathcal L}$ produces an $\text{SL}(r, {\mathbb C})$-oper
structure on this flat principal $\text{SL}(r, {\mathbb C})$-bundle. This way we obtain the
map ${\mathcal M}(r) \longrightarrow {\rm OSL}_X(r)$.

{\samepage The converse map ${\rm OSL}_X(r) \longrightarrow {\mathcal M}(r)$ is
fairly straightforward. Let $F$ be a holomorphic vector bundle of rank $r$ on $X$
such that the line bundle $\bigwedge^r F$ is holomorphically trivial, and let~$D$ be
a holomorphic connection on $F$ such that the induced connection on $\bigwedge^r F$
is the trivial one. Let ${\mathbb L} \subset F$ be a holomorphic line subbundle
such that the triple $(F, D, {\mathbb L})$ defines an~$\text{SL}(r, {\mathbb C})$-oper
on $X$. Let
\begin{equation}\label{l4}
\rho\colon\ \pi_1(X, x_0) \longrightarrow \text{SL}(F_{x_0})
\end{equation}
be the monodromy representation for $D$.}

Since $\widetilde X$ is simply connected, the pulled back connection
$(\varpi^*F, \varpi^*D)$ is trivializable, where~$\varpi$ is the projection of
$\widetilde X$ to $X$. Let $\widetilde{x}_0 \in \widetilde X$ be the base point
over the chosen base point $x_0 \in X$. So~$\varpi^*F$ is identified with the
trivial vector bundle ${\widetilde X}\times F_{x_0} \longrightarrow {\widetilde X}$
and $\varpi^*D$ is the trivial connection on it. The line subbundle
\[
\varpi^* {\mathbb L} \subset \varpi^* F = {\widetilde X}\times F_{x_0}
\]
produces a map
\[
\psi\colon\ {\widetilde X} \longrightarrow {\mathbb P}(F_{x_0});
\]
the map $\psi$ sends any $z \in {\widetilde X}$ to the line
$(\varpi^* {\mathbb L})_z \subset (\varpi^* F)_z = F_{x_0}$. Now the
pair $(\rho, \psi)$, where $\rho$ is the homomorphism in \eqref{l4}, defines an element of ${\mathcal M}(r)$.
This construction gives the opposite map ${\rm OSL}_X(r) \longrightarrow {\mathcal M}(r)$.

Using the above isomorphism ${\rm OSL}_X(r) \stackrel{\sim}{\longrightarrow} {\mathcal M}(r)$
it is straightforward to deduce that the involution $\Gamma$ of ${\rm OSL}_X(r)$
coincides with the involution $\beta$ constructed in Lemma \ref{lem3}.
\end{proof}

Proposition \ref{prop3}, Theorem \ref{prop2} and Proposition \ref{lem4}
together give the following.

\begin{Theorem}\label{cor3}
The four involutions $\Gamma$, $\beta$, $\mathcal B$ and ${\mathcal A}$ of
${\rm OSL}_X(r)$ coincide.
\end{Theorem}

\subsection*{Acknowledgements}

We thank the referees for their valuable comments.
The work of LPS is partially supported by a Simons Fellowship, NSF CAREER Award DMS 1749013 and NSF FRG Award 2152107. The material
presented here is partially based upon work supported by the National Science Foundation under Grant No. DMS-1928930 while LPS was in
residence at the Mathematical Sciences Research Institute in Berkeley, California, during the Fall 2022 semester.

\pdfbookmark[1]{References}{ref}
\LastPageEnding

\end{document}